\documentclass[11pt,reqno]{amsart}
\usepackage{amsfonts,amsmath,amssymb,graphicx,comment}
\usepackage{verbatim}
\usepackage{xcolor}
\usepackage{fancyhdr}
\usepackage[utf8]{inputenc}
\usepackage{amsmath}
\usepackage{amsfonts}
\usepackage{amssymb}
\usepackage{amsthm}
\usepackage{pictexwd, dcpic}
\usepackage{fancyhdr}
\usepackage{mathrsfs}
\usepackage{latexsym,mathtools}
\usepackage{bm, enumerate}
\usepackage{graphicx}
\usepackage{color,soul}
\usepackage{xcolor}
\usepackage{dsfont}
\usepackage{todonotes}
\usepackage{comment}

\newtheorem{thm}{Theorem}[section]
\newtheorem{cor}[thm]{Corollary}
\newtheorem{lem}[thm]{Lemma}

\newtheorem{prop}[thm]{Proposition}

\theoremstyle{definition}
     
\theoremstyle{remark}
\newtheorem{rem}[thm]{\bf{Remark}}
\numberwithin{equation}{section}

\newcommand{\beas}{\begin{eqnarray*}}
	\newcommand{\eeas}{\end{eqnarray*}}
\newcommand{\bes} {\begin{equation*}}
	\newcommand{\ees} {\end{equation*}}
\newcommand{\be} {\begin{equation}}
	\newcommand{\ee} {\end{equation}}
\newcommand{\bea} {\begin{eqnarray}}
	\newcommand{\eea} {\end{eqnarray}}

\usepackage[colorlinks=true,
    citecolor=red,
    linkcolor=blue,
    urlcolor=blue,backref=page]{hyperref}

\title{}

\begin{document}
\title[Inner multipliers, similarity and invariant subspaces]{Inner multipliers, similarity and invariant subspaces for a class of Harmonically weighted Dirichlet spaces}
\author[S. Ghara]{Soumitra Ghara}
\address{Department of Mathematics\\
Indian Institute of Technology  Kharagpur, Midnapore-721302, India}
   \email{soumitra@maths.iitkgp.ac.in\\ ghara90@gmail.com}

   \author[A. Patra]{Avipsa Patra}
\address{Department of Mathematics\\
Indian Institute of Technology  Kharagpur, Midnapore-721302, India}
   \email{avipsapatra.24@kgpian.iitkgp.ac.in}
\thanks{The work of the first named author  is supported by  ARG-MATRICS grant by the ANRF (File No: ANRF/ARGM/2025/001583/MTR) and INSPIRE Faculty Fellowship (DST/INSPIRE/04/2021/002555). The work of the second named author is supported by Institute Fellowship from IIT Kharagpur, India.}

\subjclass[2020]{}

\begin{abstract} 
 We consider the harmonically weighted Dirichlet spaces $D(\mu)$ induced by Borel measures which are mutually absolutely continuous with respect to the  Lebesgue measure $m$ on the unit circle $\mathbb T$.  Let $H^2$ and $D$ denote the Hardy space and the Dirichlet space on the unit disc $\mathbb D$, respectively. For any $f\in H^2$, let $m_f$ denote the measure defined by $dm_f(\zeta)=|f(\zeta)|^2 dm(\zeta)$. Our first result provides a characterization of functions in $D(m_f)$, when $f$ satisfies the condition $|f'(z)|^2=O(\frac{1}{(1-|z|)^{1-\epsilon}})$ for some $\epsilon>0$. We then study the multiplier algebras of these spaces with particular emphasis on inner multipliers. Specifically, when $f'$ is bounded, we obtain an explicit description of all singular inner functions in the multiplier algebra of $D(m_f)$ in terms of the radial zero set of $f$. As an application, we prove that if $f',g'$ are bounded and the radial zero sets of $f$ and $g$ are different, then the operators $(M_z, D(m_f))$ and $(M_z, D(m_g))$ are not similar. This generalizes a result of Richter showing that the operators $(M_z, D(m_{z-1}))$ and $(M_z, D)$ are not similar. Then, we classify all invariant subspaces $\mathcal M$ of $(M_z, D(\mu))$ for which  $M_z|_{\mathcal M}$ is similar to $(M_z, D(\mu))$, where  $\mu$  is any measure mutually absolutely continuous with respect to $m$. Finally, we study a special case of 
the problem of finding all functions $g\in D$ for which $gD(m_g)$ is a closed invariant subspace of $(M_z, D)$. A key ingredient in many of our results is the Richter-Sundberg formula for the local Dirichlet integral.
\end{abstract}

\maketitle

\section{Introduction}
Let $\mathbb D$ denote the open unit disc $\{z \in \mathbb C : |z| < 1\}$  and $\mathbb T$ denote the unit circle $\{z \in \mathbb C : |z| = 1\}$. Let $\mathcal{O}(\mathbb D)$ denote the space of holomorphic functions on $\mathbb D.$  Let $dA$ denote the \textit{normalized  Lebesgue area measure} on $\mathbb{D}.$ Recall that the \textit{Hardy space} $H^2$  and the \textit{Dirichlet space} $D$ are defined  by 
\begin{align*}
   &H^2=\Big\{f\in \mathcal{O}(\mathbb D): \sup_{0\leqslant r<1}\int_{0}^{2\pi}|f(re^{i\theta})|^2\frac{d\theta}{2\pi} <\infty\Big\},\\
   ~ &D=\Big\{f\in \mathcal{O}(\mathbb D): D(f):=\int_{\mathbb D}|f'(z)|^2 dA(z)<\infty\Big\},   
\end{align*}
with the norms $\|f\|_{H^2}^2:=\sup_{0\leqslant r<1}\int_{0}^{2\pi} |f(re^{i\theta})|^2 \frac{d\theta}{2\pi}$, and $\|f\|_{D}^2:=\|f\|_{H^2}^2+D(f)$, respectively.

        In \cite{richter1991representation}, Richter introduced and studied the harmonically weighted Dirichlet-type spaces  in order to classify cyclic analytic $2$-isometries. Let $M_+(\mathbb T)$ denote  the set of finite positive
Borel measures on $\mathbb T.$  For $\mu\in M_+(\mathbb T),$ let $\varphi_{\mu}(z)$ denote the positive harmonic function on $\mathbb D$ defined by $$\varphi_{\mu}(z):=\int_{\mathbb T}P(z,\eta) d\mu(\eta),$$ where $P(z,\eta)=\frac{1 - |z|^2}{|\eta - z|^2},~ z\in\mathbb D$, $\eta\in \mathbb T$, is the Poisson kernel.  For any $\mu\in M_+(\mathbb T)$ and $f\in\mathcal O(\mathbb D)$, define
$$D_{\mu}(f)=\int_{\mathbb{D}} |f'(z)|^2 \varphi_\mu(z) \, dA(z).$$
	If $\mu\neq 0$, the \textit{Dirichlet-type space} $D(\mu)$ is defined by 
		\begin{equation*}
			D(\mu)=\{f\in \mathcal{O}(\mathbb D):D_{\mu}(f) < \infty\}.
		\end{equation*}
        In this case, it turns out that $D(\mu)\subseteq H^2$, and therefore we  define a norm $\|\cdot\|_{\mu}$ on $D(\mu)$ by
        $$\|f\|_{\mu}^2=\|f\|_{H^2}^2+D_{\mu}(f).$$
		If $\mu = 0$, then $D(\mu)$ is defined to be the space $H^2$. If $\mu=m$, the normalized Lebesgue measure on $\mathbb{T}$, $D(\mu)$ becomes  the Dirichlet space $D.$ 
For any $\mu\in M_+(\mathbb T)$, $D(\mu)$ is a reproducing kernel Hilbert space of holomorphic functions on $\mathbb D$, and the polynomials are dense in $D(\mu)$. 
These $D(\mu)$ spaces have been studied extensively in the literature by several authors for general measures $\mu\in M_+(\mathbb T)$ (see, for example, \cite{richter1991representation, richter1991formula, richter1992multipliers, Shimorin, chartrand, el2014primer, BaoPou, fallah_JFA, Elfallah} and the references therein). Particular attention has been given to the case when $\mu$ is finitely supported (see \cite{Sarason},\cite{costara}). In this paper, we focus on the case when $\mu$ is mutually absolutely continuous with respect to the normalized Lebesgue measure $m$ on $\mathbb T$. 
For any $f\in H^2$, let $m_f$ denotes the measure defined by $$dm_f(\zeta):=|f(\zeta)|^2 dm(\zeta),~\zeta\in \mathbb T,$$ where 
$f(\zeta)$ is the \textit{radial limit} of $f$ at $\zeta$ defined by $f(\zeta):=\lim_{r\to 1^-}f(r\zeta)$, provided the limit exists.

Let $H_1$ and $H_2$ be  two reproducing kernel Hilbert spaces of analytic functions on $\mathbb D$. A function $\varphi \in \mathcal O(\mathbb D)$ is said to be a \emph{multiplier} from $H_1$ to $H_2$ if $\varphi f\in H_2$ for all $f\in H_1$. Let $\mbox{Mult}(H_1, H_2)$ denote the set of multipliers from $H_1$ to $H_2.$ If $H_1=H_2=H$, then we use the notation $\mbox{Mult}(H)$ instead of $\mbox{Mult}(H_1, H_1)$. 
In general, it is difficult to determine the multiplier algebras of functional Hilbert spaces. Nevertheless, these algebras have been completely characterized for several classical function spaces. For instance, $\mbox{Mult}(H^2)=H^\infty,$ the space of  bounded holomorphic functions on $\mathbb D$. A complete description of $\mbox{Mult}(D)$  was given by D. Stegenga \cite{StegengamultD} in terms of Carleson measures. In \cite{chartrand}, Chartrand characterized  the multiplier algebras of $D(\mu)$ spaces in terms of $\mu$-Carleson measures (see \cite[Theorem 2.6]{chartrand}). For the case where $\mu$ is a finitely supported measure, it is known that  $\mbox{Mult}(D(\mu))=D(\mu)\cap H^\infty.$ For a general $\mu\in M_+(\mathbb T)$, we have $\mbox{Mult}(D(\mu))\subseteq  H^\infty.$

For any $\epsilon>0$, let $\mathcal{A}_\epsilon$ denote the set of all functions $f\in\mathcal{O}(\mathbb D)$ such that
$$\sup_{z\in \mathbb D}|f'(z)|^2(1-|z|^2)^{1-\epsilon}$$ 
is finite. It is known that $\mathcal{A}_\epsilon$ is contained in the \textit{disc algebra} $A(\mathbb D)$(see \cite[Theorem 5.1]{duren1970theory}). Also, if $0<\epsilon< 1$, then $\mathcal A_{\epsilon}$ is the \textit{$\alpha$-Bloch space} with $\alpha=\frac{1-\epsilon}{2}$. 
Note that if $\epsilon>1$, then $\mathcal A_{\epsilon}$ contains only constant functions, and thus we will restrict ourselves to the case $0<\epsilon\leqslant 1$ when needed.  Our motivation to study the spaces $D(m_f)$, $f\in\mathcal A_{\epsilon}$, comes from the following result of Richter.

\begin{lem}{\rm(\cite[Lemma 7.4]{richter1991representation}\rm)}
Let $\epsilon>0$ and $f\in \mathcal{A}_\epsilon$. Then $fD(m_f)\subseteq D$  and $D\subseteq D(m_f).$
    \label{richterlipschitz}
\end{lem}
Since $D(m_f)\subseteq H^2,$ the preceding lemma  immediately implies that if $f\in \mathcal{A}_\epsilon$, then  
\begin{equation} \label{Wiener_in_RClass}
    D(m_f)\subseteq \{g\in H^2: fg\in D\}.
\end{equation}
In our first result, we prove that the reverse inclusion  of \eqref{Wiener_in_RClass} also holds, thereby  providing a characterization of functions in $D(m_f)$ when $f\in \mathcal A_{\epsilon}$.

 \begin{thm}\label{lemma4.4}
    Let $\epsilon>0$ and $f\in \mathcal{A}_\epsilon$.  Then the Dirichlet-type space $$D(m_f)=\{g\in H^2: fg\in D\}.$$  Moreover, in this case, there exist  $C_1, C_2>0$ such that the following  holds:
\begin{equation}\label{eqnequivalentnorm}
  C_1(\|g\|_{H^2}^2+D(fg))\leqslant \|g\|_{m_f}^2\leqslant C_2  (\|g\|_{H^2}^2+D(fg)),~g\in D(m_f).
\end{equation}    
\end{thm}

The assumption $f\in\mathcal A_{\epsilon}$, $\epsilon>0$, in Theorem \ref{lemma4.4} cannot, in general, be omitted. Indeed, let $f\in D\cap H^\infty$ be such that $f\notin\mbox{Mult}(D)$. By \cite[Theorem 6.2]{richter1991representation}, we have $D\subseteq D(m_f).$ However, $D(m_f)\not\subseteq \{g\in H^2:fg\in D\}.$
In fact, if the latter inclusion were to hold, then for every $g\in D(m_f)$ we would have $fg\in D$. Since $D\subseteq D(m_f)$, it would imply that $fD\subseteq D,$ and hence $f\in \mbox{Mult}(D)$, contradicting our choice of $f$. Thus, for such a function $f$, the conclusion of Theorem \ref{eqnequivalentnorm} fails. The existence of a function $f\in (D\cap H^\infty)\setminus\mbox{Mult}(D)$ is well known; see, for example, \cite[Theorem 5.1.6]{el2014primer}.

We investigate explicit descriptions of inner multipliers of  $D(m_f)$ when $f\in \mathcal A_{\epsilon}$. Recall that a function $\varphi\in H^\infty$ is called \emph{inner} if $|\varphi(\zeta)|=1$ for almost every $\zeta$ on $\mathbb T$. Every inner function $\varphi$ can be factorized into $\varphi=BS_{\sigma},$ where 
\begin{equation}\label{eqnBlaschke}
    B(z)=\gamma \prod_{n}\frac{|z_n|}{z_n}\frac{z_n-z}{1-\overline{z_n}z},~~ z\in\mathbb  D,
\end{equation} is the \textit{Blaschke product} formed with the zeros $\{z_n\}$ (finite or infinite) of $\varphi$ in $\mathbb{D},$ counting multiplicity, such that $\sum(1-|z_n|)<\infty$ and $\gamma$ is a unimodular constant, and 
\begin{equation}\label{eqnSingular}
S_{\sigma}(z)=\exp\left(-\int_\mathbb{T}\frac{e^{it}+z}{e^{it}-z}d\sigma(e^{it})\right), ~z\in\mathbb D,
\end{equation}is the \textit{singular inner function} corresponding to some $\sigma \in M_{+}(\mathbb T)$ which is singular with respect to $m.$ 
(If some $z_n=0$, we use the standard convention that $\frac{|z_n|}{z_n}=1$ for these terms.)

 It is well-known that finite Blaschke products are the only inner functions in $D$ and they also belong to
$\mbox{Mult}(D)$ (see \cite[Corollary 7.6.10 and Theorem 8.1.4]{el2014primer}) . On the contrary, this is not the case for general $D(\mu)$ spaces, see for example \cite[Theorem 3.1]{Elfallah}.  In the following theorem, we provide an explicit description of singular inner multipliers of $D(m_f)$ when $f\in \mathcal A_1.$ For any $f\in H^2$, let $Z_{\mathbb T}(f)$ denote the \emph{radial zero set of $f$} defined by  $$Z_\mathbb T(f)=\{ \zeta\in\mathbb T: \lim_{r\to 1^-} f(r\zeta)=0\}.$$ 

\begin{thm}\label{singmultiff}
Let $f\in \mathcal A_1$. 
Let $\sigma\in M_+(\mathbb T)$ be a measure singular with respect to $m$.  Then the following statements are equivalent:
    \begin{enumerate}
        \item[\rm(i)] $S_\sigma\in \mbox{Mult}(D(m_f)),$
        \item[\rm(ii)] $S_\sigma\in D(m_f),$
        \item[\rm(iii)] \mbox{supp}$(\sigma)\subseteq Z_{\mathbb T}(f).$ 
    \end{enumerate}
    \end{thm}
    We further show that, under the hypothesis of Theorem \ref{singmultiff}, the conditions (i), (ii), and (iii) are equivalent to  the condition $f^2S_{\sigma}'\in H^\infty$, see Proposition \ref{singHinftyequiv}. Moreover, the conclusion of the preceding theorem may fail in the absence of the assumption $f\in \mathcal{A}_1$. 
        More precisely, we show that, for $f\in \mathcal A_{\epsilon}$, $0<\epsilon<1,$ the membership of a singular inner function in $D(m_f)$ does not, in general, imply that it is a multiplier of $D(m_f)$, see Theorem \ref{singnotinmult}.

The proofs of Theorems \ref{singmultiff} and \ref{singnotinmult} are obtained as a consequence of a general result  which states that, if $f\in \mathcal A_{\epsilon}$, then $S_{\sigma}\in \mbox{Mult}(D(m_f))$ if and only if $S_{\sigma}\in D(m_f)$ and $f\in \mbox{Mult}(D(m_f), D(\sigma))$, see Proposition \ref{sing mult iff condn for RC}.

    We also obtain a sufficient condition for an infinite Blaschke product $B$ to be a multiplier of $D(m_f)$ when $f\in \mathcal A_1$. More specifically, we prove that if the zeros $\{z_n\}$ of  $B$ lie in a \textit{Stolz angle with vertex at $\zeta\in\mathbb T$}, then $B\in \mbox{Mult}(D(m_f))$, provided $f\in \mathcal A_1$ and $f(\zeta)=0$, see Proposition \ref{|B'(z)| bound}.
As an immediate application, we show that if $f,g\in \mathcal A_1$ and $Z_{\mathbb T}(f)\neq Z_{\mathbb T}(g)$, then 
$\mbox{Mult}(D({m_f}))\cap \mathcal S\neq \mbox{Mult}(D({m_g}))\cap \mathcal S$ and $\mbox{Mult}(D({m_f}))\cap \mathcal B\neq \mbox{Mult}(D({m_g}))\cap \mathcal B$ as sets, where $\mathcal S$ and $\mathcal B$ denote the set of all singular inner functions and all Blaschke products on $\mathbb D$, respectively.

    For $\mu\in M_+(\mathbb T)$, let $(M_z, D(\mu))$ denote the operator of multiplication by $z$ on $D(\mu).$ It turns out that $(M_z,D(\mu))$  is a cyclic analytic 2-isometry. In \cite[Theorem 6.7]{richter1991representation}, Richter 
    proved that if \((M_z, D(\mu))\) and \((M_z, D(\nu))\) are quasisimilar, then \(\mu\) and \(\nu\) are mutually absolutely continuous. 
   It was also shown that the converse of this is not true, in general (see \cite[p. 344]{richter1991representation}). But, in case $f$ is a rational function with poles off $\overline{\mathbb D}$, then $(M_z, D(m_f))$ is always quasisimilar to $(M_z, D),$ \cite[Theorem 7.5]{richter1991representation}.
  In view of this, it is therefore natural to ask if the operators $(M_z, D(m_f))$ and $(M_z, D)$ are similar or not, where $f$ is a rational function with poles off $\overline{\mathbb D}$. In the particular case, when $f(z)=z-1$, Richter proved that $(M_z,D(m_{z-1}))$ is not  similar to $(M_z, D)$, see \cite[p. 349]{richter1991representation}. We now briefly revisit his argument. Firstly, he established the following result regarding similarity between $(M_z, D)$ and its restriction to any closed invariant subspace $\mathcal M$.
 \begin{thm}{\rm(cf. \cite[Proposition 7.3]{richter1991representation}\rm)}\label{similarity to invariant subspace}
Let $\mathcal M$ be a nonzero closed invariant subspace of $(M_z, D)$. Then $M_z|_{\mathcal M}$ is similar to $(M_z, D)$ if and only if $\mathcal M=BD$ for some finite Blaschke product $B.$
\end{thm}
\noindent Moreover, it was also shown that
\begin{equation}\label{example of Richter}
    \mathcal{N}:=(z-1)S_{r\delta_1}D(m_{z-1})  \mbox{~is a closed invariant subspace of~} (M_z,D),
\end{equation}
 for any $r>0$ (see \cite[Example 7.2]{richter1991representation}). Consequently,  $(M_z, D(m_{z-1}))$ is similar to $ M_z|_{{\mathcal N}}$.   If $(M_z, D(m_{z-1}))$ is similar to $(M_z, D),$ then in view of Theorem \ref{similarity to invariant subspace}, similarity of $(M_z, D)$ and $M_z|_{{\mathcal N}}$ implies that $\mathcal{\mathcal N}=BD$ for some finite Blaschke product $B$.  Clearly, this is impossible as all functions in $\mathcal N$ have a singular inner factor $S_{r\delta_1}$.

 As an application of our results on multiplier algebras, we prove the following result on the similarity of the operators $(M_z, D(m_f))$ and $(M_z, D(m_g))$. In particular, it settles the similarity counterpart  of \cite[Theorem 7.5]{richter1991representation}.

 \begin{prop}\label{simwrtzeros}
     Let $f,g\in \mathcal{A}_1$.  If $(M_z, D(m_f))$ is  similar to $(M_z, D(m_g))$, then  $Z_\mathbb{T}(f) =  Z_\mathbb{T}(g)$. Moreover, the following statements are equivalent:
\begin{itemize}
    \item[(i)] $(M_z, D(m_f))$ is  similar to $(M_z, D)$,
    \item[(ii)] $Z_\mathbb{T}(f)=\emptyset$,
    \item[(iii)] $D(m_f)=D$ with equivalence of norms.
\end{itemize}
 \end{prop}

If $H$ is a functional Hilbert space, then  a natural question is when the restriction $M_z|_{\mathcal M}$ is similar to $M_z$ on $H$, where $\mathcal M$ is a closed invariant subspace of $(M_z, H)$. In the Hardy space $H^2$, the Beurling's Theorem implies that  $(M_z, H^2)$ is similar (in fact, unitarily equivalent) to $M_z|_{\mathcal M},$ for any nonzero closed invariant subspace $\mathcal M$ of $(M_z, H^2).$ Theorem  \ref{similarity to invariant subspace} provides an analogous result for $D.$
The following theorem generalizes this result for $D(\mu)$, where $\mu$ is mutually absolutely continuous with respect to $m.$
 \begin{thm}\label{invsubrestchar}
    Let  $\mu\in M_{+}(\mathbb T)$ be mutually absolutely continuous with respect to $m.$ Let $\mathcal M$ be a nonzero closed invariant subspace of $(M_z,D(\mu)).$ Then $M_z|_{\mathcal M}$ is similar to $(M_z,D(\mu))$ if and only if $\mathcal M=\varphi D(\mu)$ for some inner function $\varphi\in \mbox{Mult}(D(\mu)).$
\end{thm}

 If $g$ is an \textit{extremal function} in $D$, that is, $g$ is an unit vector in $D$ satisfying  $\langle g, z^ng\rangle_{D}=0$, $n\geqslant1,$ then $gD(m_g)$ is a closed invariant subspace of  $(M_z, D)$. But, in general, it is not known  for which functions $g\in D$, the space $gD(m_g)$ is a closed invariant subspace of $(M_z, D)$ (see \cite[p. 346]{richter1991representation}). In the following theorem, we consider a special case of this problem. This, in particular, generalizes the result in \eqref{example of Richter}.
 Let $\mathcal O(\overline{\mathbb D})$ denote the set of functions holomorphic in a neighbourhood of $\overline{\mathbb D}$. 

\begin{thm}\label{thminvsubspace}
    Let $f\in \mathcal O(\overline{\mathbb D})$  and $\sigma\in M_{+}(\mathbb T)$ be a   measure singular with respect to $m$. Let $g=fS_{\sigma}.$ Then $g D(m_g)$ is a closed invariant subspace of $(M_z, D)$ if and only if $\mbox{supp}(\sigma)=Z_{\mathbb{T}}(f)$ and
     all the roots of $f$ on $\mathbb T$ are simple.
    \end{thm}

 The paper is organized as follows. In section \ref{prelim}, we recall some preliminary results for $D(\mu)$ spaces. In section \ref{section description}, we prove Theorem \ref{lemma4.4}, and discuss some properties of the multiplier algebras of these spaces. In section \ref{section inner multiplier}, we study inner functions which are in the multiplier algebra of $D(m_f)$. In particular, we provide  proofs of Theorems \ref{singmultiff} and \ref{singnotinmult}. We also discuss a sufficient condition for an infinite Blaschke to be a multiplier of $D(m_f)$ when $f\in \mathcal A_1$.  In section \ref{section similarityappln},  we prove Proposition \ref{simwrtzeros} along with its consequences including its application to weakly circular cyclic analytic $2$-isometries. Proofs of Theorems  \ref{invsubrestchar} and \ref{thminvsubspace} are provided in sections \ref{similarity invariant subspace} and \ref{section invsubspace}, respectively. Finally, in section \ref{concluding remarks}, we provide an example to show that for $f,g\in\mathcal{A}_1,$ the condition $Z_{\mathbb T}(f)=Z_{\mathbb T}(g)$ is not sufficient for the equality of $\mbox{Mult}(D(m_f))\cap \mathcal B$ and $\mbox{Mult}(D(m_g))\cap \mathcal B$.

\section{Preliminaries on \texorpdfstring{$D(\mu)$}{D(mu)} spaces}\label{prelim}
In this section, we discuss some results on $D(\mu)$ spaces which will be used throughout this paper. We first set some notations. 
If $\mu=\delta_{\zeta}$, the Dirac delta measure supported at $\zeta\in \mathbb T$, we use the notation $D_{\zeta}(f)$ instead of $D_{\delta_{\zeta}}(f)$. 
We will need the following lemma which is proved in  \cite[p. 358 and Proposition 2.2]{richter1991formula} (see also \cite[Theorems 7.2.5 and 7.2.1]{el2014primer}). Its second part is known as the local Douglas formula. 
\begin{lem}\label{local Douglas}
Let    $f\in H^2$, and $\eta \in \mathbb T$. Then the following holds:
\begin{itemize}
\item[(i)] If $\alpha\in \mathbb C$ and $\int_{\mathbb{T}}\frac{|f(\zeta)-\alpha|^2}{|\zeta-\eta|^2}dm(\zeta)<\infty$, then $f(\eta)$ exists and $f(\eta)=\alpha$. 
    \item[(ii)] If $f(\eta)$ exists, then 
    $$D_{\eta}(f)=\int_{\mathbb T}\frac{|f(\zeta)-f(\eta)|^2}{|\zeta-\eta|^2} dm(\zeta).$$
    Otherwise, $D_{\eta}(f)=\infty.$
    \item[(iii)] $f\in D(\delta_{\eta})$ if and only if $f=(z-\eta)g+\alpha$ for some $g\in H^2$ and $\alpha\in\mathbb C$.
    \end{itemize}
\end{lem}

We now recall the notion of angular derivative in the sense of Carath\'{e}odory. 
Let $\zeta \in \mathbb T$, and $\kappa\in [1,\infty).$ The Stolz angle with vertex at $\zeta$ is the region defined by 
\begin{equation*}
   S_{\kappa}(\zeta):=\{z\in \mathbb D: |\zeta-z|\leqslant \kappa(1-|z|)\}. 
\end{equation*} 
For a function $f\in\mathcal O(\mathbb D),$ $L$ is called the \emph{nontangential limit} of $f$ at $\zeta,$ if $\lim_{z\to \zeta}f(z)= L$ as $z\to \zeta$ in each nontangential approach region $S_{\kappa}(\zeta)$. We say that a holomorphic self-map $\varphi$ on $\mathbb D$ has an \emph{angular derivative} in the sense of \textit{Carath\'{e}odory} at $\zeta\in\mathbb T$ if for some $\eta\in\mathbb T,$ the nontangential limit of 
     $\frac{\eta-\varphi(z)}{\zeta-z}$ exists at $\zeta.$ The limit is called the angular derivative of $\varphi$ at $\zeta$ and is denoted by $\varphi'(\zeta).$
 Whenever $\varphi$ fails to have angular derivative in the sense of
Carath\'{e}odory at $\zeta$, we write $|\varphi'(\zeta)| = \infty$. 

The following theorem which follows from \cite[Lemma 3.4, Proposition 3.5 and Theorem 3.1]{richter1991formula} (see also, \cite[Theorem 7.6.1]{el2014primer}), will be used repeatedly in this paper.

\begin{thm}[Richter--Sundberg]\label{carlrep}
  Let $\zeta\in \mathbb T$ and $\varphi$ be an inner function. Then the following holds:
  \begin{enumerate}
      \item[{\rm(i)}]For any  $f\in H^2$,  we have  
      $$ D_{\zeta}(\varphi f)= D_{\zeta}(f)+ D_{\zeta}(\varphi)|f(\zeta)|^2.$$ (In the above identity, $D_{\zeta}(\varphi f)=D_{\zeta}(f)=\infty$ whenever $f(\zeta)$ fails to exist. Moreover, if $f(\zeta)=0,$ then $D_{\zeta}(\varphi)|f(\zeta)|^2$ is interpreted as $0,$ regardless of whether $D_{\zeta}(\varphi)=\infty.$)
      \item[{\rm(ii)}] $\varphi\in D(\delta_\zeta)$ if and only if $\varphi$ has finite angular derivative at $\zeta.$ In addition, $D_\zeta(\varphi)=|\varphi'(\zeta)|. $
      \item[\rm(iii)] If $\varphi=BS_{\sigma}$, where $B$ and $S_\sigma$ are given by \eqref{eqnBlaschke} and \eqref{eqnSingular} respectively, then 
\begin{align*}
    D_{\zeta}(\varphi)&=D_{\zeta}(B)+D_{\zeta}(S_\sigma)\\
    &=\sum_n\frac{1-|z_n|^2}{|\zeta-z_n|^2}+\int_{\mathbb T}
\frac{2}{|e^{it}-\zeta|^2}
\,d\sigma(e^{it}).
\end{align*} If either of $B$ or $S_\sigma$ in the factorisation of $\varphi$ is absent, then the corresponding summand in the expression for $D_{\zeta}(\varphi)$ will be $0.$ 
  \end{enumerate}
     \end{thm}

Since  for any $\mu\in M_+(\mathbb T)$ and $g\in H^2$,  $D_{\mu}(g)=\int_{\mathbb T}D_{\zeta}(g)d\mu(\zeta)$,  Theorem \ref{carlrep} yields a formula for $D_{\mu}(\varphi f)$ for an inner function $\varphi$ and a $H^2$-function $f$. When $\mu=m,$ this reduces to a special case of Carleson's representation formula \cite{carlesonrepformula}. We also record the following consequence of Theorem \ref{carlrep} for further use:  if $f\in H^2$ and  $\varphi$ is inner, then for $\mu\in M_+(\mathbb T)$, we have
 \begin{equation}
 \|\varphi f\|_{\mu}\geqslant \|f\|_{\mu}.\label{mu norm increase}
      \end{equation}

\begin{lem}\label{RichterequalDspaces}{\rm(\cite[Corollary 6.3]{richter1991representation})}
Let $\mu, \nu\in M_+(\mathbb T)$. Then $D(\mu) = D(\nu)$  with equivalence of norms if and only if  $\nu$ and $\mu$ are mutually absolutely continuous, and the Radon-Nikodym derivative $h = \frac{d\nu}{d\mu}$ satisfies $h, \frac{1}{h} \in L^{\infty}(\mu) = L^{\infty}(\nu).$
\end{lem}

Throughout the paper, we use the notation $f \lesssim g$ (or equivalently $g \gtrsim f$) if there exists a constant $C>0$, independent of $f$ and $g$, such that
$f \leq Cg$.

%%%%%%%%%%%%%%%%%%%%%%%%%%%%%%%%%%%%%%%%%%%%%%%%%%%%%%%%%%%%%%%
\section{A characterization of  \texorpdfstring{$D(m_f)$}{D(mf)} when \texorpdfstring{$f\in \mathcal A_{\epsilon}$}{finmathcalAepsilon}}\label{section description}

For $0<\alpha\leqslant1,$ the \emph{analytic Lipschitz class} $\Lambda^{a}_{\alpha} $ is defined as the set of all functions $f\in A(\mathbb D)$ such that 
$$|f(z)-f(w)|\lesssim|z-w|^\alpha,~z,w\in \overline{\mathbb D}.$$

The following lemma provides a relation between the class $\mathcal A_{\epsilon}$ and  analytic Lipschitz class. 
\begin{lem} [Hardy-Littlewood]\label{RC_Lip_equiv}
    Let $f\in \mathcal{O}(\mathbb D)$ and $0<\epsilon \leqslant 1$. Then $f\in \mathcal{A}_{\epsilon}$ if and only if $f\in\Lambda^a_{\frac{1+\epsilon}{2}}.$
\end{lem}
\begin{proof}
    See \cite[Theorem 5.1]{duren1970theory}, see also \cite{HardyLittle}.
\end{proof}

We do not know whether the following lemma is known in the literature, so we provide a proof for completeness.

\begin{lem}\label{Wiener algebra}
    Let $0<\epsilon\leqslant 1$ and $f\in\mathcal{A}_\epsilon$. Then there exists a constant $C>0$ such that 
    \begin{equation*}
      \varphi_{m_f}(z)-|f(z)|^2\leqslant C(1-|z|^2),~z\in \mathbb D.  
    \end{equation*}
\end{lem}
\begin{proof}  
Note that
    \begin{align*}\notag
    \varphi_{m_f}(z)-|f(z)|^2&=\int_{\mathbb T}P(z,\zeta)(|f(\zeta)|^2-|f(z)|^2)dm(\zeta)\\
    &=(1-|z|^2)\int_\mathbb T \frac{(|f(\zeta)|^2-|f(z)|^2)}{|z-\zeta|^2} dm(\zeta).
    \end{align*}
This together  with   \cite[Lemma 7.3.3]{el2014primer} yields
$$\varphi_{m_f}(z)-|f(z)|^2=(1-|z|^2)\int_{\mathbb T}\frac{|f(\zeta)-f(z)|^2}{|z-\zeta|^2}dm(\zeta).$$
Since $f\in\mathcal{A}_\epsilon,$ by Lemma \ref{RC_Lip_equiv}, we have $f\in\Lambda^a_\beta,$ where $\beta=\frac{1+\epsilon}{2}.$ Moreover, since $\epsilon\in(0,1],$ we have $\beta\in(\frac{1}{2},1].$
Consequently, we have
    \begin{align*}\notag
    \varphi_{m_f}(z)-|f(z)|^2&\lesssim(1-|z|^2)\int_{\mathbb T}\frac{|z-\zeta|^{2\beta}}{|z-\zeta|^2}dm(\zeta)\\
    &=(1-|z|^2)\int_{\mathbb T}\frac{1}{|z-\zeta|^{2-2\beta}}dm(\zeta)
    \end{align*}
   Moreover, since $\beta\in(\frac 1 2, 1]$, we have $\sup_{z\in \mathbb D}\int_{\mathbb T}\frac{1}{|z-\zeta|^{2-2\beta}}dm(\zeta)$ is finite (see \cite[Theorem 1.7]{zhu}). Therefore  
    \begin{align*}
     \varphi_{m_f}(z)-|f(z)|^2&\lesssim 1-|z|^2,~z\in\mathbb D.
\end{align*}
This completes the proof.
\end{proof}

\begin{proof}[Proof of Theorem \ref{lemma4.4}]
Without loss of generality assume that $0<\epsilon\leqslant 1$. Since $f\in\mathcal{A}_\epsilon$,  
 by \eqref{Wiener_in_RClass} we have $D(m_f)\subseteq \{g\in H^2: fg\in D\}.$

  For the reverse inclusion, let $g\in H^2$ and $fg\in D$. 
By Lemma \ref{Wiener algebra}, we see that 
   \begin{align}\label{eqnD(mf)}
  \notag D_{m_f}(g)&=\displaystyle\int_{\mathbb{D}}|g'(z)|^2\varphi_{m_f}(z)dA(z)\\ \notag 
  & \lesssim \displaystyle\int_{\mathbb{D}}|g'(z)|^2(|f(z)|^2+(1-|z|^2))dA(z)\\\notag 
  &= \int_\mathbb{D}|(fg)'(z)-(f'g)(z)|^2dA(z) +\int_\mathbb{D}|g'(z)|^2(1-|z|^2)dA(z)\\\notag 
  &\lesssim \int_{\mathbb{D}}|((fg)'(z))|^2dA(z)+\int_{\mathbb{D}}|f'(z)|^2|g(z)|^2 dA(z)\\ 
  &\;\;\;\;\;\;\;\;\;\;\;\;\;\;\;\;\;\;\;\;\;\;\;\;\;\;\;\;\;\;\;\;\;\;\;\;\;\;\;\;\;\;\;\;\;\;\;\;\;\;\;+
    \int_\mathbb{D}|g'(z)|^2(1-|z|^2)dA(z).
  \end{align}
    Since for any function $h\in \mathcal O(\mathbb D)$, $\|h\|_{H^2}^2$  is equivalent to $|h(0)|^2+\int_\mathbb{D}|h'(z)|^2(1-|z|^2)dA(z)$, it follows that 
\begin{equation}\label{eqnH^2}
       \int_\mathbb{D}|g'(z)|^2(1-|z|^2)dA(z)  \lesssim  \|g\|_{H^2}^2.
\end{equation}
Also, since $H^2$ is contained in the weighted Bergman space $L^2(\mathbb D, \nu_{\epsilon})\cap \mathcal O(\mathbb D)$, where $d\nu_{\epsilon}(z)=(1-|z|^2)^{\epsilon-1}dA(z)$, we have

\begin{equation}\label{eqnA^2}
    \int_{\mathbb{D}}|f'(z)|^2|g(z)|^2 dA(z))\lesssim\int_\mathbb D|g(z)|^2(1-|z|^2)^{\epsilon-1}dA(z)\lesssim\|g\|_{H^2}^2.
\end{equation}
Using \eqref{eqnH^2} and \eqref{eqnA^2} in \eqref{eqnD(mf)}, we have 
\begin{equation*}
    D_{m_f}(g)\lesssim\|g\|_{H^2}^2+D(fg)<\infty.
\end{equation*} 
Hence $g\in D(m_f)$ establishing the reverse inclusion.

For the second part, it is easily verified that $D(m_f)$ is a reproducing kernel Hilbert space with respect to the norm $\|\cdot\|_{N}$ given by
\begin{equation*}
    \|g\|^2_{N}=\|g\|_{H^2}^2+D(fg),~g\in D(m_f).
\end{equation*}
Hence, from \cite[Theorem 5.1]{paulsen2016introduction}, the norms $\|\cdot\|_{m_f}$ and $\|\cdot\|_N$ on $D(m_f)$ must be equivalent. This yields \eqref{eqnequivalentnorm}.
    \end{proof}

\begin{prop}\label{wienermult1}
Let $\epsilon_1,\epsilon_2>0$ and $f_1\in \mathcal A_{\epsilon_1}$ be such that $f_1=f_2h$, where $f_2\in \mathcal A_{\epsilon_2}$ and $h\in H^\infty.$ Then the following statements hold: 
\begin{enumerate}
    \item[\rm(i)] if $\varphi\in H^\infty$, then $\varphi\in \mbox{Mult} (D(m_{f_1}))$ if and only if $h\varphi\in  \mbox{Mult} (D(m_{f_1}),D(m_{f_2})),$
    \item[\rm(ii)] $\mbox{Mult}(D(m_{f_2}))\subseteq \mbox{Mult}(D(m_{f_1}))$.
\end{enumerate}
\end{prop}
\begin{proof}
(i) Suppose $\varphi\in \mbox{Mult}(D(m_{f_1}))$ and $g\in D(m_{f_1}).$ Then $\varphi g\in D(m_{f_1}).$ By Theorem \ref{lemma4.4}, we have $\varphi g \in H^2$ and $f_1 \varphi g\in D.$ Since $h\in H^\infty$ and $f_1=f_2h$, we have $h\varphi g\in H^2$ and $f_2h\varphi g\in D.$  Again applying Theorem \ref{lemma4.4}, we have $h\varphi g\in D(m_{f_2}).$ Hence $h\varphi\in \mbox{Mult}(D(m_{f_1}), D(m_{f_2})).$ The converse implication follows by using similar arguments and the assumption that $\varphi\in H^\infty$. 
 
   (ii)  Suppose $\varphi\in \mbox{Mult}(D(m_{f_2}))$ and $g\in D(m_{f_1}).$ By Theorem \ref{lemma4.4},  $hg\in D(m_{f_2}).$ Thus $\varphi h g\in D(m_{f_2})$ which implies $\varphi f_1g\in D.$ Another application of Theorem \ref{lemma4.4} gives $\varphi g\in D(m_{f_1}).$ This establishes $\mbox{Mult}(D(m_{f_2}))\subseteq \mbox{Mult}(D(m_{f_1}))$. 
    \end{proof}

\begin{cor}\label{containment of Mult(D)}
Let $\epsilon>0$. If  $f\in \mathcal{A}_\epsilon$, then $$f(\mbox{Mult} (D(m_f)))\subseteq \mbox{Mult} (D)\subseteq \mbox{Mult}(D(m_f)).$$
    
\end{cor}
\begin{proof}
    Let $\varphi\in \mbox{Mult} (D(m_f))$. By part (1) of Proposition \ref{wienermult1}, we get $f\varphi\in  \mbox{Mult} (D(m_f),D)$. Also by Lemma \ref{richterlipschitz},  $D\subseteq D(m_f)$. Hence it follows that
    $f\varphi\in \mbox{Mult} (D)$, and consequently, $f(\mbox{Mult} (D(m_f)))\subseteq \mbox{Mult} (D).$ Now $ \mbox{Mult} (D)\subseteq \mbox{Mult}(D(m_f))$ is immediate from part (2) of Proposition \ref{wienermult1}.
\end{proof}
%%%%%%%%%%%%%%%%%%%%%%%%%%%%%%%%%%%%%%%%%%%%%%%%%%%%%%%%%%%%%%%%%%%%%%%%%%%%%%%%%%%%%%%%%%%%%%%%%%%%%%%%%%%%%%%%%%%%%%%%%
\section{Inner multipliers of  \texorpdfstring{$D(m_f)$}{D(mf)}}\label{section inner multiplier}
In this section, we study inner functions which are in the multiplier algebra of $D(m_f)$. We first establish several lemmas and propositions that will be used in the proof of Theorems \ref{singmultiff} and \ref{singnotinmult}.
The second part of the following lemma is recorded in \cite[Proposition 5.4]{richter1992multipliers} while the first part follows from the proof of that proposition. Although the third part can be established by an analogous argument, we include a proof for the sake of completeness.

\begin{lem}\label{prop 5.4}
Let $f\in H^2$. Let $\sigma\in M_+(\mathbb T)$ be a  measure singular with respect to $m$. Then the following holds:
\begin{itemize}
    \item[(i)] If $f(e^{it})=0$   for $\sigma$ almost every $e^{it}\in\mathbb T,$
    then
    $$D(S_{\sigma}f)=D(f)+2D_{\sigma}(f).$$
    \item[(ii)] $S_{\sigma }f\in D$ if and only if $f\in D\cap D(\sigma)$ and $f(e^{it})=0$   for $\sigma$ almost every $e^{it}\in\mathbb T.$
    \item [(iii)] $S_\sigma \in D(m_f)$ if and only if $f\in D(\sigma)$ and $f(e^{it})=0$   for $\sigma$ almost every $e^{it}\in\mathbb T.$ Moreover, in this case, $D_{m_f}(S_{\sigma})=2D_{\sigma}(f)$.
\end{itemize}
\end{lem}

\begin{proof} 
As already discussed, we provide a proof of part \rm(iii) only. By Theorem  \ref{carlrep}\rm(iii) and Fubini's theorem, we obtain
\begin{align}
    D_{m_f}(S_\sigma)&=\int_{\mathbb{T}}\int_{\mathbb{T}}\frac{2}{|e^{it}-\zeta|^2}|f(\zeta)|^2dm(\zeta)d\sigma(e^{it}).\label{D(mf)sing}
      \end{align}
    Suppose $S_\sigma \in D(m_f).$ Then by  \eqref{D(mf)sing}, we have $\int_{\mathbb{T}}\frac{|f(\zeta)|^2}{|e^{it}-\zeta|^2}dm(\zeta)$ is finite for $\sigma$ almost every $e^{it}\in\mathbb T.$ This together with Lemma \ref{local Douglas}      
    yields $f(e^{it})=0$  and  $D_{e^{it}}(f)=\int_{\mathbb{T}}\frac{|f(\zeta)|^2}{|e^{it}-\zeta|^2}dm(\zeta)$ for $\sigma$ almost every $e^{it}\in\mathbb T.$
    By another application of \eqref{D(mf)sing}, we get that $f\in D(\sigma).$ 
    The converse follows in a similar manner. Moreover, $D_{m_f}(S_{\sigma})=2D_{\sigma}(f)$ is also immediate from \eqref{D(mf)sing}.
   \end{proof}

\begin{rem}\label{suppsigmainzeroes}
Suppose $f\in A(\mathbb D)$ and  $\mu\in M_{+}(\mathbb T)$. Then, it can be shown using elementary arguments that  the condition $f(\zeta)=0$ for $\mu$ almost every $\zeta\in\mathbb T$ is equivalent to the condition that $\mbox{supp}(\mu)\subseteq Z_{\mathbb T}(f).$ 
\end{rem}
 
\begin{lem}\label{richterfuncmult}
     Let $\epsilon>0$ and $\mu\in M_{+}(\mathbb T).$ Then $\mathcal{A}_\epsilon \subseteq D(\mu).$
\end{lem}
\begin{proof}
If $\epsilon\geqslant1$, then it is trivial. If  $0<\epsilon<1$, then by \cite[Theorem 3.1]{BaoPou}, $\Lambda^a_{\frac{1+\epsilon}{2}}\subseteq D(\mu)$. An application of Lemma \ref{RC_Lip_equiv} completes the proof.
\end{proof}

\begin{lem}\label{radiallimofLip}
    Let $0<\epsilon\leqslant1,~\zeta\in\mathbb T$ and $f\in  \mathcal A_{\epsilon}$ be such that $f(\zeta)=0.$  Then for any $g\in H^2$, $(fg)(\zeta)=0.$
\end{lem}
\begin{proof}
 Since $f\in\mathcal A_\epsilon,$ Lemma \ref{RC_Lip_equiv} implies that $f\in \Lambda^a_{\beta}$ where $\beta=\frac{1+\epsilon}{2}\in(\frac{1}{2},1].$ This together with $f(\zeta)=0$ implies $|f(z)|\lesssim |z-\zeta|^\beta,~\text{for all}~ z\in\mathbb D.$ Let $g\in H^2.$ Thus we have 
\begin{equation}\label{fgrad}
    |(fg)(r\zeta)|\lesssim(1-r)^\beta|g(r\zeta)|\lesssim(1-r^2)^{\beta}|g(r\zeta)|,~ 0\leqslant r<1.
\end{equation}
Since $g\in H^2$, by \cite[Lemma 1.5.4]{el2014primer}, we have  
\begin{equation}\label{primerrad}
    \lim_{r\to1^-}(1-r^2)^{\frac{1}{2}}|g(r\zeta)|=0.
\end{equation}
Since $\beta>\frac{1}{2}$,   \eqref{fgrad} and \eqref{primerrad} imply
\begin{equation*}
    (fg)(\zeta)=\lim_{r\to 1^-}(fg)(r\zeta)=0,
\end{equation*}
which completes the proof of this lemma.
\end{proof}
\begin{prop}\label{sing mult iff condn for RC}
Let $\epsilon>0$ and $f\in\mathcal A_{\epsilon}$.  Let $\sigma\in M_{+}(\mathbb T)$ be a measure singular with respect to $m$.  
Then $S_\sigma\in \mbox{Mult}(D(m_f))$ if and only if $S_\sigma\in D(m_f)$ and  $M_f:D(m_f)\to D(\sigma)$ is a bounded linear operator.   
\end{prop}
\begin{proof}
Let  $g\in D(m_f).$ 
   Suppose $S_\sigma\in D(m_f).$ Then Lemma \ref{D(mf)sing} implies $f(\zeta)=0$  for $\sigma$ almost every $\zeta\in\mathbb T$,
    and thus by Lemma \ref{radiallimofLip}, we have $(fg)(\zeta)=0$  for $\sigma$ almost every $\zeta\in\mathbb T.$  This, combined with Lemma \ref{prop 5.4}(i) yields
    \begin{align}
    D(S_\sigma fg)=D(fg)+2D_\sigma(fg).\label{relation D with Dsigma via singmult}
    \end{align}
     For the forward implication, suppose  $S_\sigma\in\mbox{Mult}(D(m_f))$. Then $S_\sigma\in D(m_f)$ and we have $\|S_\sigma g\|_{m_f}\lesssim\|g\|_{m_f},$  for $g\in D(m_f)$. Hence, it follows from Theorem \ref{lemma4.4} that $D(fS_\sigma g)\lesssim \|g\|^2_{m_f}$. In view of \eqref{relation D with Dsigma via singmult}, we obtain $D_\sigma(fg)\lesssim \|g\|^2_{m_f}$ and consequently, 
    \begin{equation}\label{Sigma m_f}
        \|fg\|_\sigma^2=\|fg\|_{H^2}^2+D_{\sigma}(fg)\lesssim\|g\|_{m_f}.
    \end{equation}
    Hence  $M_f:D(m_f)\to D(\sigma)$ is  bounded.

    Conversely, suppose $S_\sigma\in D(m_f)$ and $M_f:D(m_f)\to D(\sigma)$ is  bounded. Thus  \eqref {Sigma m_f} holds $g\in D(m_f)$. Combining this with \eqref{relation D with Dsigma via singmult} and Theorem \ref{lemma4.4}, we obtain
    \begin{align*}
       \|S_\sigma g\|^2_{m_f}&\lesssim \|S_\sigma g\|^2_{H^2}+ D(fS_\sigma g)\\
       &=\|g\|^2_{H^2}+D(fg)+2D_\sigma(fg)\\
       &\lesssim \|g\|^2_{m_f}+\|fg\|^2_\sigma\\
       &\lesssim \|g\|^2_{m_f}.
    \end{align*}
This completes the proof.
\end{proof}
\begin{cor}\label{RCcor}
    Let $0<\epsilon\leqslant 1,~c>0,~\zeta\in\mathbb T$ and $f\in \mathcal{A_\epsilon}$. Then $S_{c\delta_\zeta}\in\mbox{Mult}(D(m_f))$  if and only if $\frac{f}{(z-\zeta)}\in\mbox{Mult}(D(m_f), H^2).$ 
\end{cor}
    \begin{proof}
Suppose $S_{c\delta_\zeta}\in\mbox{Mult}(D(m_f)).$ Then  $S_{c\delta_\zeta}\in D(m_f)$. By Proposition \ref{sing mult iff condn for RC}, $M_f:D(m_f)\to D(c\delta_{\zeta})$ is bounded. Let $g\in D(m_f).$  Then $fg\in D(c\delta_{\zeta})$. Also, by Lemma \ref{prop 5.4} and Lemma \ref{radiallimofLip}, $(fg)(\zeta)=0.$ By Lemma \ref{local Douglas}, it follows that $fg\in (z-\zeta)H^2.$
    
Conversely, assume that $\frac{f}{(z-\zeta)}\in\mbox{Mult}(D(m_f), H^2).$ Then $\frac{f}{z-\zeta}\in H^2$ and thus by Lemma \ref{local Douglas}, $f(\zeta)=0$ and $f\in D(c\delta_{\zeta})$.
Now, by Lemma \ref{prop 5.4}, $S_{c\delta_{\zeta}}\in D(m_f)$. Also $$fD(m_f)=(z-\zeta)\frac{f}{z-\zeta}D(m_f) \subseteq  (z-\zeta)H^2\subseteq D(c\delta_\zeta).$$ Hence by Proposition \ref{sing mult iff condn for RC}, it follows that $S_{c\delta_\zeta}\in\mbox{Mult}(D(m_f)).$
    \end{proof}

We are now ready to  prove Theorem \ref{singmultiff}.

\begin{proof}[Proof of Theorem \ref{singmultiff}]
    \rm(i) $\implies$ \rm(ii) This is obvious.  
    
    \rm(ii) $\implies$ \rm (iii) This follows from Lemma \ref{prop 5.4} and Remark \ref{suppsigmainzeroes}.
    
\rm(iii) $\implies$ \rm(i) Suppose $\mbox{supp}(\sigma)\subseteq Z_{\mathbb T}(f)$. Then 
 by Lemma \ref{prop 5.4} and Lemma \ref{richterfuncmult}, we have $S_\sigma\in D(m_f).$  
   In view of Proposition \ref{sing mult iff condn for RC}, in order to prove that $S_\sigma\in\mbox{Mult}(D(m_f)),$ it suffices to verify that $M_f:D(m_f)\to D(\sigma)$ is  bounded. Let $g\in D(m_f).$ We will show that $fg\in D(\sigma).$  Let $\zeta\in\mathbb T$  be such that $f(\zeta)=0.$ Application  of Lemma \ref{radiallimofLip} yields $(fg)(\zeta)=0.$
    By Theorem \ref{local Douglas}, we have 
\begin{align*}\notag
    D_\zeta(fg)&=\int_\mathbb T \frac{|(fg)(\eta)|^2}{|\eta-\zeta|^2}dm(\eta).\\\notag
   % &=\|\frac{f(z)}{z-\zeta}g(z)\|^2_{H^2}.
    \end{align*}
Since $f\in\mathcal A_1$ and $f(\zeta)=0,$ by Lemma \ref{RC_Lip_equiv}, we have $f\in A(\mathbb D)$ and $|f(\eta)|\lesssim |\eta-\zeta|,~\eta\in\mathbb T$. Hence
\begin{equation*}
    D_\zeta(fg)\lesssim \|g\|^2_{H^2}.
\end{equation*}
Moreover, since $\mbox{supp}(\sigma)\subseteq Z_{\mathbb T}(f)$, we obtain
$$D_\zeta(fg)\lesssim \|g\|^2_{H^2}, \text{ for $\sigma$ almost every $\zeta\in\mathbb T$.}$$
Consequently, we have
$$D_\sigma(fg)\lesssim \|g\|^2_{H^2}\sigma(\mathbb T)<\infty.$$ This establishes the proof. 
    \end{proof}

The following corollary  implies that $(z-1)S_{r\delta_1}\in\mathrm{Mult}(D)$ for every $r>0$. This recovers a special case of \cite[Proposition 16]{brown1984cyclic}.
\begin{cor}
    Let $\sigma\in M_+(\mathbb T)$ be a measure singular with respect to $m$. Suppose $f$ is a function in $\mathcal{A}_1$  such that 
$\mbox{supp}(\sigma)\subseteq Z_{\mathbb T}(f).$ 
Then $fS_{\sigma}\in \mbox{Mult}(D).$
\end{cor}
\begin{proof}
    This follows from  Theorem~\ref{singmultiff} and Corollary \ref{containment of Mult(D)}.
\end{proof}
    
In the following theorem, we show that the conclusion of Theorem \ref{singmultiff} may fail in the absence of the assumption $f\in\mathcal A_1.$ Note that for $\frac{1}{2}<\alpha<1$, the function $f(z)=(1-z)^{\alpha}\in \mathcal A_{2\alpha-1}$.
\begin{thm}\label{singnotinmult}
         For $\frac{1}{2}<\alpha<1$, let $f(z)=(1-z)^{\alpha}$, $z\in\mathbb D$.  Let $\sigma\in M_{+}(\mathbb T)$ be a measure singular with respect to $m$. Then the following statements hold:
         \begin{enumerate}
            \item[\rm(i)] $S_{\sigma}\in D(m_f)$ if and only if $\sigma=c\delta_1$ for some $c\geqslant 0$,
             \item[\rm (ii)] $S_{\sigma}\in \mbox{Mult}(D(m_f))$ if and only if $\sigma =0$.
         \end{enumerate}    
    \end{thm}
    We will proof Theorem \ref{singnotinmult} as a consequence of a general result on $\mbox{Mult}(D(m_f), H^2)$, $f\in  \mathcal A_{\epsilon}$.

\begin{lem}\label{kelley rep ker rel mult}
    Let $0<\epsilon\leqslant1,~\zeta\in\mathbb T$ and $f\in  \mathcal A_{\epsilon}$ be such that $f(\zeta)=0.$ If $\psi\in \mbox{Mult}(D(m_f), H^2)$, then $$\sup_{0<t<1}|\psi(t\zeta )|< \infty.$$
\end{lem}
\begin{proof}
    Let $\psi\in \mbox{Mult}(D(m_f), H^2).$ Then $M_{\psi}^*:H^2\to D(m_f)$ is bounded. Thus 
    $$\sup_{w\in\mathbb D}\frac{\|M_{\psi}^*(K^{H^2}(\cdot, w))\|_{m_f}}{\|K^{H^2}(\cdot, w)\|_{H^2}} < \infty,$$
    where $K^{H^2}(z, w)=\frac{1}{1-z\overline{w}}$, $z,w\in\mathbb D$.
    This implies 
    \begin{equation}\label{necessary}
        \sup_{w\in\mathbb D}|\psi(w)|^2(1-|w|^2) K^{m_f}(w,w)< \infty,
    \end{equation}
    where $K^{m_f}(z, w)$ is the reproducing kernel of $D(m_f)$.
    By \cite[Theorem 1]{fallah_JFA}, we have
    \begin{align}\label{eqnlip}
        K^{m_f}(w,w)\gtrsim\ 1+\int_0^{|w|}\frac{dr}{(1-r)\varphi_{m_f}(\frac{rw}{|w|})+(1-r)^2}.
    \end{align}
    Since $f\in\mathcal A_\epsilon,$ Lemma \ref{RC_Lip_equiv} implies that $f\in \Lambda^a_{\beta}$ where $\beta=\frac{1+\epsilon}{2}\in(\frac{1}{2},1].$ This together with $f(\zeta)=0$ yields
    $|f(r\zeta)|^2\lesssim (1-r)^{2\beta}.$
   Using Lemma \ref{Wiener algebra}, we have $$\varphi_{m_f}(r\zeta)\lesssim |f(r\zeta)|^2+(1-r^2)\lesssim (1-r)^{2\beta}+(1-r^2).$$
    Thus by \eqref{eqnlip}, for $t\in(0,1)$, we have 
    \begin{align*}
        K^{m_f}(t\zeta,t\zeta) &\gtrsim 1+\int_0^t \frac{dr}{(1-r)^{2\beta+1}+(1-r)^2(1+r)+(1-r)^2}\\
        & \gtrsim 1+\int_0^t \frac{dr}{(1-r)^2}\\
        &=\frac{1}{1-t}.
    \end{align*}
   Combining this  with \eqref{necessary}, we get 
    \begin{equation*}
        |\psi( t\zeta)|^2\lesssim \frac{1}{(1-t^2)K^{m_f}( t\zeta,t\zeta)}\lesssim \frac{1}{1+t}\lesssim 1,~ 0<t<1.
    \end{equation*}
    This completes the proof.
\end{proof}

\begin{proof}[Proof of theorem \ref{singnotinmult}] 
                \rm(i) This is obtained by applying Lemma \ref{prop 5.4} and Lemma \ref{richterfuncmult}.
        
        \rm(ii) Suppose $\sigma\neq 0$ and $S_{\sigma}\in\mbox{Mult}(D(m_f))$. Then $S_{\sigma}\in D(m_f),$ and by part \rm (i) of this theorem, $\sigma=c\delta_1$ for some $c> 0$. By Corollary \ref{RCcor}, we have $\frac{f}{(z-1)}\in\mbox{Mult}(D(m_f), H^2).$ 
        Thus by Lemma \ref{kelley rep ker rel mult}, $\sup_{0<t<1}|\frac{f(t)}{1-t}|< \infty,$ which is a contradiction. This completes the proof.
           \end{proof}

 The following proposition  provides a sufficient condition for an inner function to be a multiplier of $D(m_f)$.
 \begin{prop}\label{sufficient} 
Let $f\in H^2$, and $\varphi$ be an inner function. If $f^2\varphi'\in H^\infty,$ then $\varphi$ is a multiplier of $D({m_f}).$   
 \end{prop}
 \begin{proof} 
Without loss of generality, assume that $f\in H^2$ is nonzero. Suppose $f^2\varphi'\in H^\infty.$ Since $f\in H^2,$  we have $f^2\in H^1$, and thus $\varphi'\in \mathcal{N},$ where $\mathcal N$ denotes the Nevanlinna class. This guarantees that the nontangential limit $\varphi'(\zeta)$ exists $m$-a.e. on $\mathbb T.$ Since $\varphi$ is inner, the nontangential limit $\varphi(\zeta)$ exists and $|\varphi(\zeta)|=1,$ $m$-a.e. on $\mathbb T.$ Hence, by the Julia-Carath\'{e}odory theorem \cite[p. 57]{Shapiro}, it follows that the angular derivative  $\varphi'(\zeta)$ in the sense of Carath\'{e}odory exists $m$-a.e. on $\mathbb T.$ 
 Consequently, by Theorem \ref{carlrep}\rm(ii)\rm,
 \begin{equation}\label{derivative}
     D_{\zeta}(\varphi)=|\varphi'(\zeta)|,~ m\text{-a.e. on } \mathbb T.
 \end{equation}
  Let $p$ be any polynomial. By Theorem \ref{carlrep}\rm(i)   and \eqref{derivative}, we get $$ D_\zeta(\varphi p)=D_\zeta(p)+|\varphi'(\zeta)\|p(\zeta)|^2,~m\text{-a.e. on } \mathbb T.$$ Integrating we get, 
    \begin{align*}
        D_{m_f}(\varphi p)&=D_{m_f}(p)+\int_{\mathbb{T}}|\varphi'(\zeta)\|p(\zeta)|^2|f(\zeta)|^2dm(\zeta)\\
        &\lesssim D_{m_f}(p)+  \int_{\mathbb{T}}|p(\zeta)|^2 dm(\zeta) \\
        & =\|p\|_{m_f}^2. 
    \end{align*}
Since polynomials are dense in $D(m_f)$, it follows that $\varphi\in \mbox{Mult} (D(m_f)).$
\end{proof}
\begin{rem}
It is well-known that if $\varphi$ is an inner function with $\varphi'\in H^\infty$, then $\varphi$ must be a finite Blaschke product. On the other hand, $f^2\varphi'\in H^{\infty}$ (where $f\in H^2$) holds for many inner functions which are not necessarily finite Blaschke products.
\end{rem}
In general, the converse of Proposition \ref{sufficient} is not true, as $z$ is always an inner multiplier of $D(m_f)$ for any $f\in H^2$. In the following Proposition, we show that the converse of Proposition \ref{sufficient} is true when $f\in \mathcal A_1$ and $\varphi$ is singular inner.
\begin{prop}\label{singHinftyequiv}
    Let $f\in \mathcal A_1$ and $\sigma\in M_+(\mathbb T)$ be a measure singular with respect to $m$. Then the conditions $(i), (ii),  (iii)$ in Theorem \ref{singmultiff} are equivalent to $f^2 S'_\sigma\in H^\infty.$
\end{prop}
\begin{proof} In view of Proposition \ref{sufficient}, it suffices to show that condition \rm(iii)  of Theorem \ref{singmultiff} implies $f^2S'_\sigma\in H^\infty.$  Suppose 
    $\mbox{supp}(\sigma)\subseteq  Z_\mathbb{T}(f).$   
    Since $f\in \mathcal{A}_1,$ by Lemma \ref{RC_Lip_equiv}, we have for all $z\in\mathbb D,$
    \begin{equation}\label{f-bound}  
    |f(z)|\lesssim|z - \zeta|, \text{ for}~\sigma \text{ almost every } \zeta\in\mathbb T.
\end{equation}
From  \cite[Eqn.(4.9)]{mashreghi2012derivatives}, we also have that
    \begin{equation}
        S_\sigma'(z)=-S_\sigma(z)\int_\mathbb{T}\frac{2\zeta}{(\zeta - z)^2}d\sigma(\zeta),~z\in \mathbb D.\label{singfunexpression}
    \end{equation}
 Therefore, using \eqref{f-bound} and \eqref{singfunexpression}, for all $z\in \mathbb D,$ we have
   \begin{align*}  |f(z)|^2|S_{\sigma}'(z)|&\leqslant2\int_\mathbb{T}\frac{|f(z)|^2}{|\zeta-z|^2}d\sigma(\zeta)\\ &\lesssim\sigma(\mathbb T)<\infty,
   \end{align*}
   which gives the desired result. 
\end{proof}

   In the following proposition, we obtain some necessary conditions for an inner function $\varphi$ to be in $D(m_f)$. We also show that these conditions are sufficient if $f\in \mathcal{A}_\epsilon$, for some $\epsilon>0$. We point out that  similar results for Blaschke products also appeared in \cite[Theorem 3.1]{Elfallah}.
    \begin{prop}\label{innerfunclassification}
    Let $f\in H^2$. Let $\varphi$ be an inner function of the form $BS_\sigma$. If $\varphi\in D(m_f),$ then $\sum_n |f(z_n)|^2 < \infty$ and $f(\zeta)=0 \text{ for}~\sigma$ almost every $\zeta\in\mathbb T.$ 
    If, in addition, $f\in \mathcal{A_\epsilon}$ for some $\epsilon>0,$ then the above conditions are also sufficient.
\end{prop}
\begin{proof}
       By Theorem \ref{carlrep}\rm(iii) and Fubini's theorem, we obtain 
    \begin{align}\notag
     D_{m_f}(BS_\sigma)
     &=D_{m_f}(B)+D_{m_f}(S_\sigma)\\\notag
     &=\int_\mathbb{T}\sum_n\frac{1-|z_n|^2}{|\zeta-z_n|^2}|f(\zeta)|^2dm(\zeta)+D_{m_f}(S_\sigma)\\
     &=\sum_n \varphi_{m_f}(z_n)+ D_{m_f}(S_\sigma)\label{phi_m_f}\\
     &\geqslant \sum_n|f(z_n)|^2+D_{m_f}(S_\sigma).\label{phi_mf>|f|^2}
    \end{align}
    Suppose $\varphi=BS_\sigma\in D(m_f).$ \eqref{phi_mf>|f|^2} together with Lemma \ref{prop 5.4} yields that $\sum_n |f(z_n)|^2<\infty$ and $f(\zeta)=0 \text{ for}~\sigma$ almost every $\zeta\in\mathbb T.$

For the converse, suppose $f\in \mathcal{A}_\epsilon$ for some $\epsilon>0$ and assume that $\sum_n |f(z_n)|^2<\infty$ and $f(\zeta)=0 \text{ for}~\sigma$ almost every $\zeta\in\mathbb T.$
Thus from Lemmas \ref{richterfuncmult} and \ref{prop 5.4}  , $S_\sigma\in D(m_f).$
    Using Lemma \ref{Wiener algebra} and \eqref{phi_m_f}, it follows that
    \begin{equation*}
        D_{m_f}(BS_\sigma)\lesssim \sum_n |f(z_n)|^2+\sum_n(1-|z_n|^2)+D_{m_f}(S_\sigma)<\infty.
    \end{equation*}
    Hence $\varphi\in D(m_f).$
        \end{proof}

     \begin{lem}\label{BlaschkeinMult}
     Let $\zeta\in\mathbb T,\text{and }\kappa\in[1,\infty). $ Let $B$ be a Blaschke product with zeros $\{z_n\}$ such that $\{z_n\}\subseteq S_{\kappa}(\zeta).$ Then there exists a $C>0$ such that
     $$|B'(z)|\leqslant \frac{C}{|\zeta-z|^{2}}, ~z\in \mathbb D.$$
     \end{lem}
 \begin{proof}
     This follows from \cite[Eqn. (2.3)]{Girela}.
 \end{proof}
 We now provide a sufficient condition for an infinite Blaschke $B$ to be a multiplier of $D(m_f).$            

 \begin{prop}\label{|B'(z)| bound}
    Let $\zeta\in \mathbb T$ and  $f\in \mathcal{A}_1$ be such that $f(\zeta)=0.$ Suppose $B$ is a Blaschke product with zeros $\{z_n\}\subseteq S_{\kappa}(\zeta).$ Then $B\in \mbox{Mult}(D(m_f)).$ 
 \end{prop}
 \begin{proof}
     By Lemma \ref{BlaschkeinMult}, we have $(z-\zeta)^2B'\in H^\infty$.  
      Since $f\in \mathcal{A}_1$ and $f(\zeta)=0$, by Lemma \ref{RC_Lip_equiv}, we get $\frac{f}{z-\zeta}\in H^\infty$. Thus
      $$f^2B'=\frac{f^2}{(z-\zeta)^2}(z-\zeta)^2B'\in H^\infty.$$  Hence, by Proposition \ref{sufficient}, $B\in\mbox{Mult}(D(m_f)).$
     \end{proof}

\begin{cor}\label{Blaschkemultcor}
 Let $\zeta\in\mathbb T$ and $f\in \mathcal{A}_1$.  Then $f(\zeta)=0$ if and only if 
there exists a Blaschke product $B\in \mbox{Mult}(D({m_f}))$ with zeros $\{z_n\}$ such that $z_n\to \zeta$.
   \end{cor}
\begin{proof}

Suppose $f(\zeta)=0$. Let $z_n=\zeta(1-\frac{1}{n^2}),~n\in\mathbb N.$ Then, clearly, $z_n\to \zeta$, $\sum_n(1-|z_n|)<\infty$ and it is easy to verify that $\{z_n\}\subseteq S_\kappa(\zeta)$ for any $\kappa \geqslant1$. Let $B$ be the infinite Blaschke with zeros $\{z_n\}.$ Hence, by Proposition  \ref{|B'(z)| bound}, $B\in \mbox{Mult}(D(m_f)).$
      The converse follows from Proposition \ref{innerfunclassification}, combined with the fact that $f\in A(\mathbb D)$. 
\end{proof}

%%%%%%%%%%%%%%%%%%%%%%%%%%%%%%%%%%%%%%%%

%%%%%%%%%%%%%%%%%%%%%%%%%%%%%%%%%%%%%%
%%%%%%%%%%%%%%%%%%%%%%%%%%%%%%%%%%%%%%%%%%%%%%%
\section{Similarity of the operators \texorpdfstring{$(M_z, D(m_f))$}{(Mz,D(mf))} via multiplier algebras}\label{section similarityappln}
In this section, using the results of section \ref{section inner multiplier}, we provide a proof of Proposition \ref{simwrtzeros}.  
Recall that $\mathcal S$ and $\mathcal B$ denote the set of singular inner functions and Blaschke products on $\mathbb D$, respectively. 

 \begin{thm}\label{diffmultalg}
                 Let  $f,g\in \mathcal{A}_1$. If  $Z_{\mathbb T}(f)\neq Z_{\mathbb T}(g)$, then 
                 \begin{enumerate}
                     \item[\rm(i)] $\mbox{Mult}(D({m_f}))\cap \mathcal S$ and $\mbox{Mult}(D({m_g}))\cap \mathcal S$ are not equal as sets,
                     \item[\rm(ii)] $\mbox{Mult}(D({m_f}))\cap \mathcal B$ and $\mbox{Mult}(D({m_g}))\cap \mathcal B$ are not equal  as sets.
                 \end{enumerate}
                 In particular, if $Z_{\mathbb T}(f)\neq Z_{\mathbb T}(g)$, then  $\mbox{Mult}(D({m_f}))$ and $\mbox{Mult}(D({m_g}))$ are not equal as vector spaces.
    \end{thm}
    
    \begin{proof}
\rm(i) Suppose $Z_{\mathbb T}(f)\neq Z_{\mathbb T}(g)$. Then, either $Z_{\mathbb T}(f)\setminus Z_{\mathbb T}(g)$  or  $Z_{\mathbb T}(g)\setminus Z_{\mathbb T}(f)$ is nonempty. Suppose $Z_{\mathbb T}(f)\setminus Z_{\mathbb T}(g)$ is nonempty, and let $\zeta \in Z_{\mathbb T}(f)\setminus Z_{\mathbb T}(g).$  Then applying Theorem \ref{singmultiff} twice, we get $S_{\delta_{\zeta}}\in \mbox{Mult}(D({m_f}))$ and $S_{\delta_{\zeta}}\notin \mbox{Mult}(D({m_g}))$.

\rm(ii) This follows using similar arguments as in \rm(i) together with Corollary \ref{Blaschkemultcor}.
    \end{proof}

\begin{rem}
    In view of Theorem \ref{singmultiff}, it is evident that the converse of Theorem \ref{diffmultalg}(i) also holds,  that is, if $f,g\in\mathcal{A}_1$  with $Z_{\mathbb T}(f)=Z_{\mathbb T}(g)$, then $\mbox{Mult}(D({m_f}))\cap \mathcal S$ and $\mbox{Mult}(D({m_g}))\cap \mathcal S$ are equal as sets. In section \ref{concluding remarks}, we will see that the converse of Theorem \ref{diffmultalg}(ii) does not hold true in general.
\end{rem}
\begin{rem}
Since Dirichlet-type spaces have the complete Nevanlinna-Pick property (see \cite{Shimorin}), it is worth pointing out that a result of Hartz (\cite[Corollary 3.2]{Hartz Multiplier algebra}) implies that for $\mu, \nu\in M_+(\mathbb T)$,  $\mbox{Mult}(D(\mu))=\mbox{Mult}(D(\nu))$ isometrically if and only if $D(\mu)=D(\nu)$ with equality of norms. In view of \cite[Theorem 4.1]{richter1991representation}, this is further equivalent to  $\mu=\nu$.
\end{rem}

 The following lemma is well known, see, for example \cite[section 10.3]{alemanhartz}. It provides a relation between similarity of multiplication operators $M_z$ on two reproducing kernel Hilbert spaces and the corresponding multiplier algebras.  
 \begin{lem}\label{similaritymult}
  Let $H_1$ and $H_2$ be two reproducing kernel Hilbert spaces of holomorphic functions on $\mathbb D$ such that polynomials are dense in both $H_1$ and $H_2$. Assume that $M_z$ is bounded on both $H_1$ and $H_2$. If $(M_z, H_1)$ is similar to $(M_z, H_2)$, then 
  $\mbox{Mult}(H_1)= \mbox{Mult}(H_2)$ as vector spaces.
\end{lem}

\begin{rem}
   Note that the converse of the above lemma is not true in general. For example, consider the weighted Bergman spaces $H^{(\lambda)}$ on $\mathbb D$ determined by the positive definite kernel $(1-z\overline{w})^{-\lambda}$, $\lambda \geqslant 1.$  It is well known that $\mbox{Mult}(H^{(\lambda)})=H^{\infty}$ for each $\lambda \geqslant 1$. But if $\lambda_1\neq \lambda_2$, then from 
    \cite[Theorem $2^\prime$]{Shields}, it follows that the operators $(M_z, H^{(\lambda_1)})$ and $(M_z, H^{(\lambda_2)})$ are not similar.   
\end{rem}

\begin{proof}[Proof of Proposition \ref{simwrtzeros}]
Suppose $Z_\mathbb T(f)\neq Z_{\mathbb{T}}(g)$ and $(M_z, D(m_f))$ is similar to $(M_z, D(m_g))$. Then, by Theorem \ref{diffmultalg}, $\mbox{Mult}(D(m_f))$ and $\mbox{Mult}(D(m_g))$ are not equal as vector spaces which is a contradiction by Lemma \ref{similaritymult}.

For the second part, by Lemma \ref{RichterequalDspaces} and the fact that $f\in A(\mathbb D)$, \rm(ii) and \rm(iii) are equivalent.   The implication \rm(i) $\implies$ \rm(ii) follows from the first part of this theorem. For \rm(iii) $\implies$ \rm(i),  note that if $D(m_f)=D$ with equivalence of norms, then the identity operator from $D(m_f)$ to $D$ is an invertible  bounded linear operator which intertwines $(M_z, D(m_f))$ and $(M_z, D)$. 
\end{proof}

\begin{cor}\label{corrational}
    Let $f,g$ be two rational functions on $\mathbb D$ with poles off $\overline{\mathbb D}$. If the roots of both $f$ and $g$ on $\mathbb T$ are simple, then the following statements are equivalent:
    \begin{enumerate}
        \item [\rm(i)\rm] $\mbox{Mult}(D(m_f))=\mbox{Mult}(D(m_g))$ as vector spaces,
        \item [\rm(ii)\rm] $Z_{\mathbb T}(f)=\mathbb Z_{\mathbb T}(g),$
        \item [\rm(iii)\rm] $D(m_f)=D(m_g)$ with equivalence of norms,
        \item[\rm(iv)\rm] $(M_z,D(m_f))$ is similar to $(M_z,D(m_g)).$
        \end{enumerate}
\end{cor}
\begin{proof}
The implication \rm(i) $\implies$ \rm(ii) follows from Theorem \ref{diffmultalg}.
 Since the roots of $f$ and $g$ on $\mathbb{T}$ are simple, by Lemma \ref{RichterequalDspaces}, \rm(ii) $\implies$ \rm(iii) holds. The implication  \rm(iii) $\implies$ \rm(iv) is obvious. Finally, 
\rm(iv) $\implies$ \rm(i) follows from Lemma \ref{similaritymult}.
\end{proof}

\begin{cor}
    Let $f\in \mathcal{A}_1.$  The following statements are equivalent:
\begin{itemize}
    \item [\rm(i)\rm] $\mbox{Mult}(D(m_f))=\mbox{Mult}(D)$ as vector spaces,
    \item[(ii)] $Z_\mathbb{T}(f)=\emptyset$,
    \item[(iii)] $D(m_f)=D$ with equivalence of norms,
    \item[(iv)] $(M_z, D(m_f))$ is similar to $(M_z, D)$.
\end{itemize}
\end{cor}
\begin{proof}
The proof follows using similar arguments used in the proof of Corollary \ref{corrational}.
\end{proof}

\begin{cor}
    Let $f\in\mathcal{A}_1$ be such that $f(1)=0$ and $g(z)=(1-z)^\alpha,~\frac{1}{2}<\alpha<1.$ Then $(M_z,D(m_f))$ is not similar to $(M_z,D(m_g).$
\end{cor}
\begin{proof}
    The proof follows from Theorems \ref{singmultiff}, \ref{singnotinmult} and Lemma \ref{similaritymult}.
    \end{proof}

\subsection{Weakly circular cyclic, analytic \texorpdfstring{$2$}{2}-isometries} 
In this subsection, we study circularity and weak circularity of the operator $(M_z, D(\mu)).$  Following \cite{AHHK}, an operator $T$ on a Hilbert space $H$ is called  \emph{circular} if $T$ is unitarily equivalent to $\lambda T$ for all $\lambda\in \mathbb T$. The operator $T$ is said to be \emph{weakly circular} if $T$ is similar to $\lambda T$ for all $\lambda\in \mathbb T$. We start with the following proposition which classifies all $\mu$ in $M_+(\mathbb T)$ for which $(M_z, D(\mu))$ is circular. Although some of the arguments used in the proof of the following proposition are well known (see, for example, \cite{GKT}), we provide the proof for reader's convenience. 

\begin{prop}
    Let $\mu\in M_+(\mathbb T)$. The operator $(M_z, D(\mu))$ is circular if and only if $\mu=cm$ for some $c\geqslant 0.$
\end{prop}
\begin{proof}
Suppose  $\mu=cm$. Then monomials are orthogonal in $D(\mu)$, and therefore $(M_z, D(\mu))$ is unitarily equivalent to a unilateral weighted shift. Hence by \cite[Proposition 1.2]{AHHK}, $(M_z, D(\mu))$ is circular. For the converse, 
    let $(M_z, D(\mu))$ be circular. Let $\lambda\in \mathbb T$ and $U_{\lambda}$ be a unitary operator such that $U_{\lambda}M_z=\lambda M_z U_{\lambda}.$ Since $\ker M_{z}^*=\mbox{span} \{1\}$, we have $U_{\lambda}^*(1)=\eta$ for some $\eta\in \mathbb T$. Therefore, for any non-negative integers $m,n$, we have 
    \begin{align*}
        \langle z^n, z^m\rangle_{D(\mu)}= \langle M_z^n(1), M_z^m(1)\rangle_{D(\mu)}&= \langle \lambda^nU_{\lambda}^*M_z^nU_{\lambda}(1),\lambda^m U_{\lambda}^*M_z^mU_{\lambda}(1)\rangle_{D(\mu)}\\
        &=\lambda^{n-m}\langle z^n, z^m \rangle_{D(\mu)}.
    \end{align*}
      This yields $\langle z^n, z^m \rangle_{D(\mu)}=0$ for $m\neq n.$ This together with the fact that  $$\langle z^n, z^m \rangle_{D(\mu)} = \mbox{min}\{m,n\}\int_{\mathbb T}\zeta^{n-m} d\mu(\zeta), m\neq n$$ implies $\int_{\mathbb T}\zeta^{k} d\mu(\zeta)=0$ for all nonzero integers $k$. Hence $\mu$ is of the form $cm$ for some $c\geqslant 0$.
\end{proof}

It is easily verified that every operator $T$ which is similar to a circular operator is weakly circular.   The following corollary shows that, up to similarity, the Dirichlet shift is the only weakly circular operator within a class of cyclic analytic $2$-isometries.

\begin{cor}\label{themweaklycircular}
Let $f\in \mathcal{A}_1$ be a nonzero function and $\mu\in M_+(\mathbb T)$. Suppose $\mu=\mu_s+m_f,$ where $\mu_s$ is singular with respect to $m$. Then $(M_z, D(\mu))$ is weakly circular if and only if $\mu_s=0$ and  $Z_{\mathbb T}(f)=\emptyset.$ In this case, $D(\mu)=D$ with equivalence of norms, and therefore, $(M_z, D(\mu))$ is similar to  $(M_z, D).$
\end{cor}
We need few lemmas to prove Corollary \ref{themweaklycircular}.
 For any  $\mu\in M_+(\mathbb T)$ and $\lambda\in\mathbb T,$ let $\mu_{\lambda}\in M_{+}(\mathbb T)$ denote the measure defined by
            \begin{equation*}
                \mu_{\lambda}(\Delta)=\mu(\overline{\lambda}\Delta),~\mbox{for any Borel set} ~ \Delta\subseteq \mathbb T.
            \end{equation*}
\begin{lem}\label{lefthaarquas}{\rm(\cite[Proposition 11]{Bourbaki})}
            Let $\mu\in M_+(\mathbb T)$.
   If $\mu_{\lambda}$ and $\mu$ are mutually absolutely continuous for all $\lambda\in\mathbb T$, then $\mu$ is mutually absolutely continuous with respect to $m$.
        \end{lem}

\begin{lem}
{\rm(\cite[Proposition 7.1]{Ghara})\label{studiamath}}
     Let $\mu\in M_+(\mathbb T)$ and $\lambda\in\mathbb T.$  Let $T=(M_z, D(\mu)).$ Then the operator $\lambda T$ is unitarily equivalent to $(M_z, D(\mu_{\lambda}))$.
\end{lem}

    \begin{lem}\label{weakcirleb}
           Let $\mu\in M_{+}(\mathbb T).$ If $(M_z,D(\mu))$ is weakly circular, then $\mu$ is mutually absolutely continuous with respect to $m$.
        \end{lem}
        \begin{proof}
            Suppose $(M_z,D(\mu))$ is weakly circular.
            %Then $(M_z,D(\mu))$ is similar to $(e^{i\theta} M_z,D(\mu))$ for all $e^{i\theta}\in\mathbb T$. 
            Then by Lemma \ref{studiamath},  $(M_z, D(\mu))$ is similar to $(M_z, D(\mu_{\lambda}))$, for each $\lambda\in\mathbb T$. By \cite[Theorem 6.7]{richter1991representation}, $\mu$ and $\mu_\lambda$ are mutually absolutely continuous. 
By Lemma \ref{lefthaarquas}, it follows that  $\mu$ is mutually absolutely continuous with respect to $m$.
\end{proof}    

\begin{proof}[Proof of Corollary \ref{themweaklycircular}]
Suppose $\mu_s=0$ and $Z_{\mathbb T}(f)=\emptyset.$  By Proposition  \ref{simwrtzeros}, we have $D(\mu)=D$ with equivalence of norms and $(M_z, D(\mu))$ is similar to $(M_z, D),$ which is a circular operator. Hence $(M_z, D(\mu))$ is weakly circular.

  Conversely, suppose $(M_z, D(\mu))$ is weakly circular. By Lemma \ref{weakcirleb}, $\mu$ is mutually absolutely continuous with respect to $m$. Consequently, $\mu_s=0$ and $\mu=m_f.$ Let $T=(M_z, D(m_f)).$ By Lemma \ref{studiamath},  $\lambda T$ is unitarily equivalent to $(M_z, D(m_{f_{\lambda}}))$ for each $\lambda\in\mathbb T,$ where $f_\lambda(z)=f(\overline \lambda z),~z\in \mathbb D.$
   Thus $T$ is similar to $(M_z, D(m_{f_{\lambda}})),$ for each $\lambda\in\mathbb T.$ Hence, by Proposition  \ref{simwrtzeros}, $Z_\mathbb T (f)=Z_\mathbb T (f_\lambda),$ for each $\lambda \in \mathbb T.$
   Since $f$ is a non-zero function in $A(\mathbb D),$ it follows that $Z_\mathbb T(f)=\emptyset.$

 The second part  follows immediately from the proof of the first part.
   \end{proof}
   %%%%%%%%%%%%%%%%%%%%%%%%%%%%%%%%%
   %%%%%%%%%%%%%%%%%%%%%%%%%%%%%%%%%%%%%
\section{Similarity between \texorpdfstring{$(M_z,D(\mu))$}{D(mu)} and its restriction to invariant subspaces}\label{similarity invariant subspace}
In this section, we study the similarity of $(M_z, D(\mu))$ with its restriction to invariant subspaces, where  $\mu\in M_{+}(\mathbb T)$ is any measure mutually absolutely continuous with respect to $m.$ We shall make use of the following lemma to prove our theorem. 

\begin{lem}\label{invertible multiplier}
 Let $\mu\in M_+(\mathbb T)$. Suppose  $f\in \mbox{Mult}(D(\mu))$ and $\frac{1}{f}\in H^\infty.$ Then $\frac{1}{f}\in \mbox{Mult}(D(\mu)).$
\end{lem}
\begin{proof}
    Let $g\in D(\mu).$  It suffices to show that $D_{\mu}(\frac{g}{f})<\infty.$ Since $f\in \mbox{Mult}(D(\mu))$, we have $f\in H^\infty$. Note that
    
    \begin{align}\notag
        D_\mu\left(\frac{g}{f}\right)=\int_{\mathbb D}|\left(\frac{g}{f}\right)'|^2\varphi_{\mu}(z)dA(z)
        &=\int_{\mathbb D}\left|\frac{fg'-f'g}{f^2}\right|^2\varphi_\mu(z)dA(z)\\\notag
        &=\int_{\mathbb D}\left|\frac{2fg'-(fg)'}{f^2}\right|^2\varphi_\mu(z)dA(z)\\\notag
        &\lesssim \int_{\mathbb D}|{2fg'-(fg)'}|^2\varphi_\mu(z)dA(z)\\\notag
        &\lesssim\int_{\mathbb D}|f|^2|g'|^2\varphi_\mu(z)dA(z)+ \int_{\mathbb D}|{(fg)'}|^2\varphi_\mu(z)dA(z)\\\notag
        &\lesssim D_{\mu} (g)+ D_\mu (fg)<\infty.
        \end{align}
        Hence $\frac{1}{f}\in \mbox{Mult}(D(\mu))$.
\end{proof}

We are now ready to prove Theorem \ref{invsubrestchar}. The proof is motivated by \cite[Proposition 7.3]{richter1991representation}. Note that, since multiplication by an inner function increases the $D(\mu)$ norm (see \eqref{mu norm increase}), $\varphi D(\mu)$ is a closed invariant subspace of $(M_z, D(\mu))$ for any inner function $\varphi\in\mbox{Mult}(D(\mu)).$ 
\begin{proof}[Proof of Theorem \ref{invsubrestchar}]
    Suppose $\mathcal M=\varphi D(\mu)$, where $\varphi$ is an inner function in $\mbox{Mult}(D(\mu)).$ Then one can easily verify that the operator of multiplication by $\varphi$ serves as an invertible intertwiner between $M_z|_{\mathcal M}$ and $(M_z,D(\mu))$. 
    
    To prove the converse, assume that $M_z|_{\mathcal M}$ is similar to $(M_z,D(\mu)).$ Thus, by standard arguments, the similarity must be given by an invertible multiplication operator $M_\psi:D(\mu)\to \mathcal M.$  This yields $\psi\in \mbox{Mult}(D(\mu))$ and $\mathcal M=\psi D(\mu).$  Let $\psi=\psi_{i}\psi_{o}$ be the inner-outer factorization of $\psi.$ 
    By \eqref{mu norm increase},  it follows that $\psi_{o}\in \mbox{Mult}(D(\mu)).$ 
     Since $\psi\in \mbox{Mult}(D(\mu))$ and $M_\psi$ is bounded below, by \cite[Corollary 6.6]{richter1991representation}, $\frac{1}{\psi}\in L^{\infty}(\mu)=L^{\infty}(m).$ Moreover, since 
    $|\psi|=|\psi_{o}|~m$-a.e. on $\mathbb T,$ we have $\frac{1}{\psi_{o}}\in L^{\infty}(m).$ By Smirnov's maximum principle (see \cite[Theorem 2.11]{duren1970theory}), $\frac{1}{\psi_o}\in H^\infty.$ Since, $\psi_o\in \mbox{Mult}(D(\mu))$ and $\frac{1}{\psi_{o}}\in H^\infty,$  by Lemma \ref{invertible multiplier}, $\frac{1}{\psi_o}\in\mbox{Mult}(D(\mu)).$ 
    Therefore, $M_{\psi_o}$ is an invertible operator on $D(\mu)$, and $\psi_o D(\mu)=D(\mu).$
    We also have $\psi_{i}=\frac{\psi}{\psi_{o}}\in \mbox{Mult}(D(\mu)).$ 
   Consequently, $$\mathcal M=\psi D(\mu)=\psi_i \psi_{o}D(\mu)=\psi_i D(\mu),$$ 
 with $\psi_i \in \mbox{Mult}(D(\mu)).$ This completes the proof.
\end{proof}
%%%%%%%%%%%%%%%%%%%%%%%%%%%%
%%%%%%%%%%%%%%%%%%%%%%%%%%
%%%%%%%%%%%%%%%%%%%%%%%%%%%%%%%%%%%%%%%%%
    \section{A class of closed invariant subspaces of the Dirichlet space}\label{section invsubspace}
    In this section, we provide a proof of Theorem \ref{thminvsubspace}. 
    We shall make use of the following lemma.
    \begin{lem}\label{notcyclic}
        Let $\lambda\in\mathbb T$. There does not exist any $C>0$ such that 
        \begin{equation}\label{eqninvariant}
            \|p\|_{H^2} \leqslant C \|(z-\lambda)p\|_D~\mbox{ for all polynomials } p. 
        \end{equation}
    \end{lem}
    \begin{proof}
        If possible, assume that there exists a $C>0$ such that \eqref{eqninvariant} holds for all polynomials $p$. Since $(z-\lambda)$ is cyclic in $D$ (see \cite[Lemma 8]{brown1984cyclic}), there exists a sequence of polynomials $\{p_n\}$ such that $(z-\lambda)p_n\to1$ in $D$. Thus $\{p_n\}$ is Cauchy in $H^2$ by \eqref{eqninvariant}, and therefore there exists a $h\in H^2$ such that $p_n \to h $ in $H^2.$  Since evaluations are continuous in $H^2$ and $D$, it follows that $(w-\lambda)h(w)=1$, $w\in\mathbb D.$ This is a contradiction as $h\in H^2.$
\end{proof}

\begin{proof}[Proof of Theorem \ref{thminvsubspace}]
We first consider the case $Z_{\mathbb T}(f)=\emptyset$. Assume that $\mbox{supp}(\sigma)=Z_{\mathbb  T}(f).$ Thus $g=f$ and $gD(m_g)=fD(m_f)$. Moreover, by Lemma \ref{RichterequalDspaces}, it is easy to verify that $gD(m_g)= BD$ for some finite Blaschke product $B$ and therefore it is closed. Conversely, if $gD(m_g)$ is a closed invariant subspace of $(M_z,D),$ then $g=fS_\sigma\in D$ and by Lemma \ref{prop 5.4} and Remark \ref{suppsigmainzeroes}, $\mbox{supp}(\sigma)=\emptyset.$  Hence the theorem follows.

Now we assume that $f$ has at least one root on $\mathbb T$. Suppose all the roots of $f$ on $\mathbb T$ are simple and $\mbox{supp}(\sigma)=Z_{\mathbb{T}}(f)$. Let $$f(z)=(z-e^{i\phi_1})(z-e^{i\phi_2})\ldots(z-e^{i\phi_n})h(z),$$ where $\phi_1,\ldots,\phi_n$ are distinct real numbers in $[0,2\pi)$ and $h\in\mathcal{O}(\overline{\mathbb D})$ with $Z_\mathbb{T}(h)=\emptyset$. Note that $m_f=m_g$. Application of  Theorem \ref{singmultiff} and Lemma \ref{richterlipschitz} yields $gD(m_g)\subseteq D$. This together with the fact that the polynomials are dense in $D(m_g)$, it is easily verified that $gD(m_g)$ is closed in $D$ if 
\begin{equation}\label{closedinvequnpoly}
    \|q\|_{m_g}\lesssim\|g q\|_D, \text{  for all polynomials } q.
\end{equation}
Let $\sigma=\sum_{j=1}^{n}c_j\delta_{e^{i\phi_j}}$  with $c_j>0$ for $j=1,\ldots,n.$ By Theorem \ref{carlrep}, 
\begin{align}\notag
    D(gq)=D(fS_\sigma q)&= D(fq)+2\int_\mathbb{T}\int_\mathbb T  \frac{|f(\zeta)|^2|q(\zeta)|^2}{|\zeta-e^{i\phi}|^2}d\sigma(e^{i\phi})dm(\zeta)\\
    &=    D(fq)+ 2\sum_{j=1}^nc_j\int_\mathbb T  \frac{|f(\zeta)|^2|q(\zeta)|^2}{|\zeta-e^{i\phi_j}|^2}dm(\zeta)\label{eqninvariant2}.
\end{align}
We claim that there exists a $C>0$ such that $\sum_{j=1}^n c_j \frac{|f(\zeta)|^2}{|\zeta-e^{i\phi_j}|^2} > C$ for all $\zeta\in \mathbb T.$ If not, then there exists a $\theta_0\in [0,2\pi)$ such that 
\[
|h(e^{i\theta_0})|^2\big(\sum_{j=1}^n c_j \prod_{k \neq j} |e^{i\theta_0} - e^{i\phi_k}|^2\big) = 0. \quad 
\]
Since $\mathbb Z_{\mathbb T}(h)=\emptyset$, it follows that 
    $ \prod_{k \neq j} |e^{i\theta_0} - e^{i\phi_k}|^2=0$ for all $j=1,\ldots,n.$ This implies that $\theta_0\in \{\phi_1,\ldots,\phi_n\}\setminus \{\phi_j\}$ for all $j$. This is a contradiction since $\phi_1,\ldots,\phi_n$ are distinct. Hence the claim is established.
    Now, from \eqref{eqninvariant2} and Theorem \ref{lemma4.4}, for all polynomials $q$, we have 
\begin{align*}
    \|gq\|^2_{D}=\|fS_\sigma q\|^2_D\geqslant D(fS_{\sigma}q)&\geqslant D(fq)+ C\|q\|^2_{H^2}\\ &\gtrsim D(fq)+ \|q\|^2_{H^2}\\
   & \gtrsim \|q\|^2_{m_f}\\
   &= \|q\|^2_{m_g},
\end{align*}  
establishing \eqref{closedinvequnpoly}.

To prove the forward implication, suppose that  $g D(m_g)$ is a closed invariant subspace of $(M_z,D).$  Note that since $g\in D,$ by Lemma \ref{prop 5.4} and Remark \ref{suppsigmainzeroes}, we have $\mbox{supp}(\sigma)\subseteq Z_{\mathbb T}(f).$
If possible, assume that all the roots of $f$ on $\mathbb T$ are not simple. Let 
\begin{equation}\label{polyform}
    f(z)=(z-e^{i\phi_1})^{m_1}(z-e^{i\phi_2})^{m_2}\ldots(z-e^{i\phi_n})^{m_n}h(z),
\end{equation} where $\phi_1,\ldots,\phi_n$ are distinct real numbers in $[0,2\pi)$, $h\in \mathcal O(\overline{\mathbb D}),$  $Z_{\mathbb T}(h)=\emptyset,$ and $m_1,\ldots,m_n \geqslant 1$ are positive integers with at least one $m_{j}\geqslant 2$, say  $m_{\ell}.$
Then for all $j=1,\ldots, n$, we have
\begin{equation}\label{eqnextrafactor}
    \frac{|f(\zeta)|^2}{|\zeta-e^{i\phi_j}|^2}=|\zeta-e^{i\phi_\ell}|^2|f_j(\zeta)|^2,~\zeta\in\mathbb T
\end{equation}
for some $f_j\in \mathcal{O}(\overline{\mathbb D})$. Let  $\sigma=\sum_{j=1}^n c_j\delta_{e^{i\phi_j}}$, where $c_j\geqslant 0.$ 
Applying Theorem \ref{carlrep}, we have
$$D(gq)=D(fS_\sigma q)=D(fq)+2\sum_{j=1}^{n}c_j\int_\mathbb{T}   \frac{|f(\zeta)|^2|q(\zeta)|^2}{|\zeta-e^{i\phi_j}|^2}dm(\zeta). $$
Combining this with \eqref{eqnextrafactor}, we get
\begin{align}
   \notag D(fS_\sigma q)& \leqslant D(fq)+\Big(2\sum_{j=1}^n c_j \|f_j\|_{\infty}^2\Big) \int_\mathbb T |\zeta-e^{i\phi_\ell}|^2|q(\zeta)|^2dm(\zeta),\\
    &\lesssim D(fq)+\|(z-e^{i\phi_\ell})q\|^2_{H^2}.\label{Singcycliccont}
\end{align}
Since $gD(m_g)$  is a closed invariant subspace of $D$, for all polynomials $q$, we have  
\begin{equation*}
    \|q\|^2_{m_g}\lesssim\|g q\|^2_D=\|gq\|^2_{H^2}+D(g q)=\|fq\|^2_{H^2}+D(fS_\sigma q).
\end{equation*}
This together with \eqref{Singcycliccont} and Theorem \ref{lemma4.4} yields, for all polynomials $q$
\begin{align}\label{closed}
 \|q\|_{m_g}^2=\|q\|_{m_f}^2= \|q\|^2_{H^2}+ D(fq)\lesssim \|fq\|^2_{D}+\|(z-e^{i\phi_\ell})q\|^2_{H^2}.
\end{align} 
Since $\frac{f}{z-e^{i\phi_\ell}}\in \mbox{Mult}(D),$ we get
\begin{equation}
    \|fq\|^2_D\lesssim\|(z-e^{i\phi_\ell})q\|^2_D, \text{ for all polynomials $q$}.\label{A2}
\end{equation}
Also, since $D$ is contractively contained in $H^2$, we have
\begin{equation}
   \|(z-e^{i\phi_\ell})q\|^2_{H^2}\leqslant \|(z-e^{i\phi_\ell})q\|^2_{D}, \text{ for all polynomials $q$}.\label{A3}
\end{equation}
Using \eqref{A2} and \eqref{A3} in \eqref{closed}, we have for all polynomials $q,$
$$\|q\|^2_{H^2}\lesssim \|(z-e^{i\phi_\ell})q\|^2_D,$$ which is a contradiction by Lemma \ref{notcyclic}. Hence all the roots of $f$ on $\mathbb T$ are simple.

To complete the proof, let us now assume that $\mbox{supp}(\sigma)\subsetneq Z_{\mathbb{T}}(f).$ 
Suppose $f$ is of the form \eqref{polyform}, where $\phi_1,\ldots,\phi_n$ are distinct real numbers in $[0,2\pi)$,  $Z_{\mathbb T}(h)=\emptyset$ and $m_1,\ldots,m_n \geqslant 1$ are positive integers.
 Since $\mbox{supp}(\sigma)\subsetneq Z_{\mathbb{T}}(f),$ there exists $1\leqslant\ell\leqslant n$ such that $e^{i\phi_\ell}\notin \mbox{supp}(\sigma).$
Let $\sigma=\sum_{j=1,j\neq \ell}^nc_j\delta_{e^{i\phi_j}},$ where $c_j\geq 0, j\neq \ell.$ Then for all $j\neq \ell$, we have 
\begin{equation*}
    \frac{|f(\zeta)|^2}{|\zeta-e^{i\phi_j}|^2}=|\zeta-e^{i\phi_\ell}|^2|f_j(\zeta)|^2,~\zeta\in\mathbb T,
\end{equation*}
for some  $f_j\in \mathcal{O}(\overline{\mathbb D}).$
By repeating the previous argument, we arrive at a contradiction. Hence $\mbox{supp}(\sigma)=Z_{\mathbb T}(f).$ This completes the proof.
\end{proof}

\begin{cor}
    Let $f\in \mathcal O(\overline{\mathbb D})$ be a function with simple roots on $\mathbb T$. Then there exists a closed invariant subspace $\mathcal M$ of $(M_z, D)$ such that $(M_z, D(m_f))$ is similar to  $M_z|_{\mathcal M}$. 
\end{cor}
\begin{proof}
By Theorem \ref{thminvsubspace}, $\mathcal M:=fS_{\sigma}D(m_f)$ is a closed invariant subspace of $(M_z, D)$, where $\sigma$ is any singular measure with $\mbox{supp}(\sigma)=Z_{\mathbb T}(f)$. The similarity of  $M_z|_{\mathcal M}$ and $(M_z, D(m_f))$ now follows as the operator $M_{fS_{\sigma}}:D(m_f)\to \mathcal M$ serves as an invertible intertwiner. 
\end{proof}
%%%%%%%%%%%%%%%%%%%%%%%%
%%%%%%%%%%%%%%%%%%%%%%%%
\section{Concluding remarks}\label{concluding remarks}
We finish this paper by showing that the converse of Theorem \ref{diffmultalg}\rm(ii) is not true in general, that is, for $f,g\in \mathcal A_1$, $Z_{\mathbb T}(f)=Z_{\mathbb T}(g)$ may not imply that $\mbox{Mult}(D(m_f))\cap \mathcal B$ and $\mbox{Mult}(D(m_g))\cap\mathcal B$ are equal. To see this, we recall the definition of  \emph{Stolz-type} regions $S_{\kappa, \gamma}(\zeta)$ (see \cite[p. 126]{AhernClark}). For $\zeta \in \mathbb T$, and  $\gamma,\kappa\in [1,\infty),$  the Stolz-type region is defined by
\begin{equation*}
S_{\kappa, \gamma}(\zeta):=\{z\in \mathbb D: |\zeta-z|^\gamma\leqslant \kappa(1-|z|)\}. 
\end{equation*}

We will need the following generalization of Lemma \ref{BlaschkeinMult}. 
 \begin{lem}{\rm(\cite[Theorem 1]{matstud})}\label{HO_BlaschkeinMult}
     Let $\zeta\in\mathbb T,\gamma,\kappa\in[1,\infty). $ Let $B$ be a Blaschke product with zeros $\{z_n\}$ such that $\{z_n\}\subseteq S_{\kappa, \gamma}(\zeta).$ Then there exists a $C>0$ such that
     $$|B'(z)|\leqslant \frac{C}{|\zeta-z|^{2\gamma}}, ~z\in \mathbb D.$$
    
 \end{lem}

\begin{prop}\label{HOB_D(mp)}
    Let $m\geqslant 1 $ be a positive integer, $\zeta\in\mathbb T$ and $p$ be a polynomial such that $p^{(k)}(\zeta)=0$ for $k=0,1,\ldots,m-1$. Let $B$ be a Blaschke product with zeros $\{z_n\}$ such that $\{z_n\}\subseteq S_{\kappa, m}(\zeta).$ Then $B\in \mbox{Mult}(D(m_p)).$
\end{prop}
\begin{proof}
   By assumption, $\frac{p^2}{(z-\zeta)^{2m}}\in H^\infty.$ By Lemma \ref{HO_BlaschkeinMult}, $(z-\zeta)^{2m}B'\in H^\infty.$ Consequently, $p^2B'\in H^\infty,$ and therefore, by Proposition \ref{sufficient}, $B\in\mbox{Mult}(D(m_p)).$
\end{proof}

For any $\zeta\in \mathbb T,$ consider $f(z)=z-\zeta$ and $g(z)=(z-\zeta)^4.$ Clearly, $Z_{\mathbb T}(f)=Z_{\mathbb T}(g).$
Let $B$  be an infinite Blaschke product with zeros $z_n=\zeta(1-\frac{1}{n^2})e^{\frac{i}{\sqrt{n}}}, n\in\mathbb N.$ By a straightforward computation we obtain 
\begin{equation}
    |z_n-\zeta|^2= 4(1-\frac{1}{n^2})\sin^2\frac{1}{2\sqrt{n}}+\frac{1}{n^4}.\label{conrem}
\end{equation}
Since the series $\sum_n (1-\frac{1}{n^2})\sin^2 \frac{1}{2\sqrt{n}}$ diverges, so does the series $\sum_n |\zeta-z_n|^2.$ Hence by Proposition \ref{innerfunclassification}, $B\notin D(m_{z-\zeta}).$ Consequently, $B\notin \mbox{Mult}(D(m_{z-\zeta})).$ 

Moreover, from \eqref{conrem} we have 
$$|z_n-\zeta|^2\leqslant \frac{1}{n}(1-\frac{1}{n^2})+\frac{1}{n^4}\leqslant\frac{1}{n}. $$
Thus we have that $|z_n-\zeta|^4\leqslant\frac{1}{n^2}=1-|z_n|,$ which implies that $\{z_n\}\subseteq S_{1, 4}(\zeta).$ By Proposition \ref{HOB_D(mp)}, $B\in \mbox{Mult}(D(m_{(z-\zeta)^4})).$ Thus, $\mbox{Mult}(D(m_{z-\zeta}))\cap\mathcal B$ and  $\mbox{Mult}(D(m_{(z-\zeta)^4}))\cap\mathcal{B}$ are not equal. Therefore, in view of Lemma \ref{similaritymult}, we obtain the following result.
\begin{cor}
    For any $\zeta\in\mathbb T$, $\mbox{Mult}(D(m_{z-\zeta}))$ and   $\mbox{Mult}(D(m_{(z-\zeta)^4}))$ are not equal as vector spaces. Consequently,  the operators $(M_z, D(m_{z-\zeta}))$ and $(M_z, D(m_{(z-\zeta)^4}))$ are not similar.
\end{cor}

\vspace{0.1in}
\noindent\textbf{Data Availability:}
Data sharing is not applicable to this article, as no datasets were generated or analyzed during
the current study. In case any datasets are generated during and/or analyzed during the current study, they must
be available from the corresponding author on reasonable request.

\vspace{0.1in}
\noindent\textbf{Conflict of interest:} There is no conflict of interest.

\end{document}